%% file: main_DualityPassivity_V2_arXiv.tex
\documentclass[a4paper]{article}

\usepackage{amssymb,amsmath, amsfonts, theorem}
\usepackage{graphicx}
\usepackage{epsfig}
\usepackage{algorithm,algorithmic,xspace}
\usepackage{times}
\usepackage{color}
\usepackage{flafter}
\usepackage{psfrag}
\usepackage{pstricks}
\usepackage[tight,footnotesize]{subfigure}
\usepackage{bbm}
\usepackage{float}

\usepackage{pgfplots}
\pgfplotsset{compat=1.3}

\usepackage[Symbol]{upgreek}

\usepackage{tikz}
\usetikzlibrary{matrix}
\usetikzlibrary{positioning}
\usetikzlibrary{arrows,shapes}

\usepackage{pst-pdf}

\usepackage{subfigure}

\usepackage{cases}
\usepackage{enumerate}

\usepackage{authblk}
\usepackage{fullpage}

\newcommand{\Edgeset}{\mathbf{E}}
\newcommand{\Nodeset}{\mathbf{V}}

\newcommand{\1}{\boldsymbol{ \mathbbm{1}}}

\newcommand{\xr}{\mathrm{x}}
\newcommand{\yr}{\mathrm{y}}
\newcommand{\ur}{\mathrm{u}}
\newcommand{\vr}{\mathrm{v}}

\renewcommand{\wr}{\mathrm{w}} 
 \renewcommand{\vr}{\mathrm{v}} 
 
 \newcommand{\wb}{\mathbf{w}} 
 
\newcommand{\xb}{\mathbf{x}}
\newcommand{\xd}{\boldsymbol{x}}

\newcommand{\yb}{\mathbf{y}}
\newcommand{\yd}{\boldsymbol{y}}

\newcommand{\ub}{\mathbf{u}}
\newcommand{\ud}{\boldsymbol{u}}

\newcommand{\vb}{\mathbf{v}}
\newcommand{\vd}{\boldsymbol{v}}

\newcommand{\zb}{\boldsymbol{\upeta}}
\newcommand{\zd}{\boldsymbol{\eta}}
\newcommand{\zr}{\upeta}
\newcommand{\z}{\eta}

\newcommand{\mub}{\boldsymbol{\upmu}} 
\newcommand{\mud}{\boldsymbol{\mu}} 
\newcommand{\mur}{\upmu}

\newcommand{\zetab}{\boldsymbol{\upzeta}} 
\newcommand{\zetad}{\boldsymbol{\zeta}} 
\newcommand{\zetar}{\upzeta}
 
 \newcommand{\gammar}{\upgamma}

\newcommand{\kyi}[1]{k_{\yr,#1}}
\newcommand{\kyf}{\mathbf{k}_{\yr}}

 \newcommand{\bfpsi}{ \boldsymbol{ \psi} }
 \newcommand{\fb}{\boldsymbol{f}}
 \newcommand{\hb}{\boldsymbol{h}}

 \newcommand{\mc}{\mathcal}

 \newtheorem{theorem}{Theorem}[section]
 \newtheorem{proposition}[theorem]{Proposition}
 \newtheorem{corollary}[theorem]{Corollary}
 \newtheorem{definition}[theorem]{Definition}
 \newtheorem{lemma}[theorem]{Lemma}
 \newtheorem{remark}[theorem]{Remark}

 \newtheorem{assumption}[theorem]{Assumption}
 
\newcommand{\bea}{\begin{eqnarray}}
\newcommand{\eea}{\end{eqnarray}}
\newcommand{\beas}{\begin{eqnarray*}}
\newcommand{\eeas}{\end{eqnarray*}}
\newcommand{\leftm}{\left[\begin{array}}
\newcommand{\rightm}{\end{array}\right]}

\newcommand{\margin}[1]{\marginpar{\tiny\color{red} #1}}

\newcommand{\MBremove}[1]{\margin{removed by MB}}

\newcommand\oprocendsymbol{}
\newcommand\oprocend{\relax\ifmmode\else\unskip\hfill\fi\oprocendsymbol}

\graphicspath{{figs/}}

\begin{document}

\title{Duality and Network Theory in \\Passivity-based Cooperative Control}
\author[1]{Mathias B\"{u}rger}
\author[2]{Daniel Zelazo}
\author[1]{Frank Allg{\"o}wer}
\affil[1]{Institute for Systems
Theory and Automatic Control, University of Stuttgart, Pfaffenwaldring 9, 70550 Stuttgart, Germany {\tt \{mathias.buerger, frank.allgower\}@ist.uni-stuttgart.de}
}
\affil[2]{Faculty of Aerospace Engineering, Technion - Israel Institute of Technology, Haifa 32000, Israel {\tt dzelazo@technion.ac.il}.}

\maketitle

\begin{abstract}
This paper presents a class of passivity-based cooperative control problems that have an explicit connection to convex network optimization problems. 
The new notion of maximal equilibrium independent passivity is introduced and it is shown that networks of systems possessing this property asymptotically approach the solutions of a dual pair of network optimization problems, namely an \emph{optimal potential} and an \emph{optimal flow} problem. 
This connection leads to an interpretation of the dynamic variables, such as system inputs and outputs, to variables in a network optimization framework, such as divergences and potentials,
and reveals that several duality relations known in convex network optimization theory translate directly to passivity-based cooperative control problems.
The presented results establish a strong and explicit connection between passivity-based cooperative control theory on the one side and network optimization theory on the other, and
they provide a unifying framework for network analysis and optimal design. 
The results are illustrated on a nonlinear traffic dynamics model that is shown to be asymptotically clustering.
\end{abstract}

%
%
%
%

\section{Introduction}

One of the most profound concepts in mathematics is the notion of \emph{duality}.  In many ways, duality theory is the mathematical answer to the idiom ``there are two sides to every coin."  This powerful concept manifests itself across many mathematical disciplines, but perhaps the most elegant and complete notion of duality is the celebrated Lagrange duality in convex optimization theory \cite{Boyd2003}. 
One of the most complete expositions of this duality theory relates to a class of optimization problems over networks, generally known as \emph{network optimization} \cite{Rockafellar1998}, \cite{Bertsekas1998}.
In \cite{Rockafellar1998}, a unifying framework for network optimization was established, with the key elements being a pair of dual optimization problems: the \emph{optimal network flow problem} and the \emph{optimal potential problem}. This dual pair of optimization problems characterizes the majority of network decision problems. 
The notion of duality also has a long history within the theory of control systems, as control problems are often intimately related to optimization problems and their respective duals have again a controls interpretation, see e.g. \cite{Balakrishnan1995}. 

A recent trend in modern control theory is the study of cooperative control problems amongst groups of dynamical systems that interact over an information exchange network.  A fundamental goal for the analysis of these systems is to reveal the interplay between properties of the individual dynamic agents, the underlying network topology, and the interaction protocols that influence the functionality of the overall system  \cite{Mesbahi2010}. 
%
%
Amongst the numerous control theoretic approaches being pursued to define a general theory for networks of dynamical systems, \emph{passivity} \cite{Willems1972} takes an outstanding role; see e.g., \cite{Bai2011}. 
The conceptual idea underlying passivity-based cooperative control is to separate the network analysis and synthesis into two layers. 
On the systems layer, each dynamical system comprising the network is designed to have a certain input-output behavior, namely passivity.
Then, the complete network can be analyzed by considering only the input-output behavior of the individual systems and the network topology describing their interconnections. 
This conceptual idea was proposed in \cite{Arcak2007}, where a passivity-based framework for group coordination problems was established. Following this, passivity was used in \cite{Zelazo2009b} to derive performance bounds on the input/output behavior of consensus-type networks. 
Passivity is also widely used in coordinated control of robotic systems \cite{Chopra2006} and the teleoperation of UAV swarms \cite{Franchi2011}. Passivity-based cooperative control with quantized measurements is, for example, studied in \cite{DePersis2012a}.
The related concepts of incremental passivity or relaxed co-coercivity have been used to study various synchronization problems, see \cite{Stan2007} or \cite{Scardovi2010}, respectively. 
Passivity was also used in the context of Port-Hamiltonian systems on graphs in \cite{Schaft2012}.
A passivity framework was used to study \emph{clustering} in networks of scalar dynamical systems with saturated couplings \cite{Burger2011}, \cite{Burger2011a}, \cite{Burger2012}.

In \cite{Hines2011}, a new definition of passivity was introduced that serves the needs of networked systems. In particular, \emph{equilibrium independent passivity} was introduced to characterize dynamical systems that are passive with respect to an arbitrary equilibrium point that can be generated by a constant input signal. Equilibrium independent passivity enables a convergence analysis of dynamical networks without computing the convergence point a priori. A similar passivity concept can be found in \cite{Jayawardhana2007}.

The passivity-based cooperative control framework and the network optimization framework have many modeling similarities, as both rely on a certain matrix description of the underlying network (i.e., the incidence matrix). However, to the best of our knowledge, an explicit connection between these two research areas has not yet been established, and this motivates the main thesis of this work:  \textit{Does the cooperative control framework inherit any of the duality results found in network optimization}?  

We present in this paper a class of networks consisting of dynamical systems with certain passivity properties that are intimately related to the network optimization theory of \cite{Rockafellar1998} and admit similar duality interpretations. Our results build an analytic bridge between cooperative control theory and network optimization theory.

The contributions are as follows.
Building upon the existing work on passivity-based cooperative control and in particular on  \cite{Hines2011}, we introduce a refined version of equilibrium independent passivity, named \emph{maximal equilibrium independent passivity}. 
The new definition is motivated by the fact that the original definition excludes some important systems, such as integrators.
Maximal equilibrium independent passive systems are passive for different steady-state input-output configurations and the equilibrium input-output maps are allowed to be \emph{maximal monotone relations}.
A cooperative control framework involving maximal equilibrium independent passive systems that aim to reach output agreement is considered. 
First, necessary conditions for an output agreement solution to exist are derived, considering only the systems dynamics and the network topology. 
It it then shown that any steady-state configuration that satisfies these necessary conditions is \emph{inverse optimal}, in the sense that the corresponding steady-state input solves a certain optimal flow problem and the steady-state output solves the corresponding dual optimal potential problem.
Exploiting this connection, certain results on the existence and uniqueness of output agreement solutions are derived using tools from convex optimization theory.
Following this, conditions on the couplings are derived that ensure the output agreement steady-state is realized. Maximal equilibrium independent passivity turns out to be again the central concept, and by using convex analysis, conditions are derived that ensure that the desired steady-state input can be generated for all possible network configurations.
The dynamic state and the output variable of the couplings are also shown to be inverse optimal with respect to a dual pair of network optimization problems.

Following the discussion on output agreement problems, the inverse optimality and duality results are generalized to a broader class of networks of maximal equilibrium independent passive systems.
The general results are used to analyze a nonlinear traffic dynamic model, that is shown to exhibit asymptotically a clustering behavior.

The reminder of the paper is organized as follows. 
The network optimization framework of \cite{Rockafellar1998} is reviewed in Section \ref{sec.NetworkTheory}.
In Section \ref{sec.Problem} the dynamical network model is introduced, some results on passivity-based cooperative control are reviewed and the new notion of maximal equilibrium independent passivity is introduced.
The connection to network optimization theory is established in Section \ref{sec.OutputAgreement}. First, necessary conditions for the existence of an output agreement solution are derived and inverse optimality and duality results for those solutions are presented.
The inverse optimality results are then generalized to networks of maximal equilibrium independent passive systems in Section \ref{sec.NetworkAnalysis}.
The theoretical results are illustrated on a nonlinear traffic dynamics model in Section \ref{sec.Applications}.  

\subsection*{Preliminaries}
A function $\phi : \mathbb{R}^{n} \mapsto \mathbb{R}^{n}$ is said to be
 \emph{strongly monotone} on $\mc{D}$ if there exists $\alpha >0$ such that
 $ (\phi(\eta) - \phi(\xi))^{\top}(\eta - \xi) \geq \alpha \|\eta - \xi\|^{2}$
 for all $\eta, \xi \in \mc{D}$, and
 \emph{co-coercive} on $\mc{D}$ if there exists $\gamma >0$ such that
 $ (\phi(\eta) - \phi(\xi))^{\top}(\eta - \xi) \geq \gamma \|\phi(\eta) - \phi(\xi)\|^{2}$
 for all $\eta, \xi \in \mc{D}$, see, e.g., \cite{Zhu1995}. 
A function $\Phi : \mathbb{R}^{q} \mapsto \mathbb{R}$ is said to be 
\emph{convex} on a convex set $\mc{D}$ if for any two points $\eta, \xi \in \mc{D}$ and for all $\lambda \in [0,1]$,
  $ \Phi(\lambda \eta + (1-\lambda)\xi) \leq \lambda \Phi(\eta) + (1-\lambda) \Phi(\xi).$
  It is said to be \emph{strictly convex} if the inequality holds strictly and  \emph{strongly convex} on $\mc{D}$ if there exists $\alpha > 0$ such that
  for any two points $\eta, \xi \in \mc{D},$ with $\eta \neq \xi$, and for all $\lambda \in [0,1]$
   $$ \Phi(\lambda \eta + (1-\lambda)\xi) < \lambda \Phi(\eta) + (1-\lambda) \Phi(\xi) -  \frac{1}{2} \lambda (1-\lambda)\alpha \|\eta - \xi\|^{2}.$$

The \emph{convex conjugate} of a convex function $\Phi$, denoted $\Phi^{\star}$, is defined as \cite{Rockafellar1997}:
\begin{align} \label{eqn.Conjugate}
\Phi^{\star}(\xi) = \sup_{\eta \in \mc{D}} \{ \eta^{\top} \xi - \Phi(\eta) \} = - \inf_{\eta \in \mc{D}} \{ \Phi(\eta) - \eta^{\top} \xi \}.
\end{align}
The definition of a convex conjugate implies that for all $\eta$ and $\xi$ it holds that $\Phi(\eta) + \Phi^{\star}(\xi) \geq \eta^{\top}\xi.$ 
A vector $g$  is said to be a \emph{subgradient} of a function $\Phi$ at $\eta$ if $\Phi(\eta') \geq \Phi(\eta) + g^{\top}(\eta' - \eta)$. 
The set of all subgradients of $\Phi$ at $\eta$ is called the \emph{subdifferential} of $\Phi$ at $\eta$ and is denoted by $\partial \Phi(\eta)$. The multivalued mapping $\partial \Phi :  \eta \rightarrow \partial \Phi(\eta)$ is called the subdifferential of $\Phi$, see \cite{Rockafellar1997}.

A special convex function we employ is the \emph{indicator function}. Let $\mc{C}$ be a closed, convex set, the indicator function is defined as 
\begin{align*}
I_{\mathbb{C}}(\eta) = \begin{cases}
0 & \mbox{if}\; \eta \in \mc{C} \\[-0.5em]
+\infty & \mbox{if}\; \eta \notin \mc{C}.
\end{cases}
\end{align*} 
We will also use the indicator function for points, e.g., $I_{0}(\eta)$ as the indicator function for $\eta = 0$.

Given a control system $\dot{x} = f(x,u)$, with state $x \in \mathbb{R}^{p}$ and input $u \in \mathbb{R}^{q}$, and a function $S(x)$ mapping $\mathbb{R}^{p}$ to $\mathbb{R}$, the \emph{directional derivative} of $S$ is denoted by $\dot{S} =  \frac{\partial S}{\partial x} f(x,u)$.

%
%
%
%
%
%
%

\section{Network Optimization Theory} \label{sec.NetworkTheory}

The objective of this paper is to study passivity-based cooperative control in the context of \emph{network optimization theory} \cite{Rockafellar1998}. 
A \emph{network} is described by a \emph{graph} $\mc{G}=(\Nodeset,\Edgeset)$, consisting of a finite set of
\emph{nodes}, $\Nodeset=\{v_1,\ldots, v_{|\Nodeset|}\}$ and a finite set of \emph{edges}, $\Edgeset=\{e_1, \ldots, e_{|\Edgeset|}\}$, describing the incidence relation between
pairs of nodes. 
Although we consider $\mc{G}$ in the cooperative control problem as an \emph{undirected graph}, we assign to each edge an arbitrary orientation. The notation $e_{k} = (v_{i},v_{j}) \in \Edgeset \subset \Nodeset \times \Nodeset$ indicates that $v_{i}$ is the initial node of edge $e_{k}$ and $v_{j}$ is the terminal node. For simplicity, we will abbreviate this with $k = (i,j)$, and write $k \in \Edgeset$ and $i,j \in \Nodeset$.

The \emph{incidence matrix} $E \in \mathbb{R}^{|\Nodeset | \times |\Edgeset |}$ of the
graph $\mc{G}$ with arbitrary orientation, is a $\{0, \pm 1\}$ matrix with the rows and columns indexed
by the nodes and edges of $\mc{G}$ such that $[E]_{ik}$ has value
`+1' if node $i$ is the initial node of edge $k$, `-1' if it is the terminal
node, and `0' otherwise. This definition implies that for any graph, $\1^{\top} E = 0$, where $\1 \in \mathbb{R}^{|\Nodeset|}$ is the vector of all ones. 
We refer to the \emph{circulation space} of $\mc{G}$ as the null space $\mc{N}(E)$, and the \emph{differential space} of $\mc{G}$ as the range space $\mc{R}(E^{\top})$; see \cite{Rockafellar1998}. 
Additionally, we call $\mc{N}(E^{\top})$ the \emph{agreement space}. 
Note that $\mc{N}(E^{\top})\perp\mc{R}(E)$ and $\mc{N}(E)\perp\mc{R}(E^{\top})$.

We call a vector $\mub = [\mur_{1}, \ldots, \mur_{|\Edgeset|}]^{\top} \in \mathbb{R}^{|\Edgeset|}$ a \emph{flow} of the network $\mc{G}$. An element of this vector, $\mur_{k}$, is the \emph{flux} of the edge $k \in \Edgeset$.
The incidence matrix can be used to describe a type of conservation relationship between the flow of the network along the edges and the net in-flow (or out-flow) at each node in the network, termed the \emph{divergence} of the network $\mc{G}$.  The net flux entering a node must be equal to the net flux leaving the node.  The divergence associated with the flow $\mub$ is denoted by the vector $\ub = [\ur_{1}, \ldots, \ur_{|\Nodeset|}]^{\top} \in \mathbb{R}^{|\Nodeset|}$ and can be represented as\footnote{This condition is \emph{Kirchhoff's Current Law}.}
\begin{align}\label{flow_conservation}
\ub + E\mub = 0.
\end{align}

Borrowing from electrical circuit theory, we call the vector $\yb \in \mathbb{R}^{|\Nodeset|}$ a \emph{potential} of the network $\mc{G}$. 
%
To any edge $k = (i,j)$, one can associate the \emph{potential difference} as
$
\zetar_{k} = \yr_{j} - \yr_{i};
$
we also call this the \emph{tension} of the edge $k$. The \emph{tension vector} $\zetab = [\zetar_{1}, \ldots, \zetar_{|\Edgeset|}]^{\top}$, can be expressed as\footnote{This condition is \emph{Kirchoff's Voltage Law}}
\begin{align}\label{network_tension}
\zetab = E^{\top}\yb.
\end{align}
Flows and tensions are related to potentials and divergences by the \emph{conversion formula} 
$
\mub^{\top}\zetab = -\yb^{\top}\ub.
$
Network theory broadly connects elements of graph theory to a family of convex optimization problems. 
 The beauty of this theory is that it admits elegant and simple duality relations.

The \emph{optimal flow problem} attempts to optimize the flow and divergence in a network subject to the conservation constraint (\ref{flow_conservation}).  Each edge is assigned a \emph{flux cost}, $C_k^{flux}(\mur_k)$, and each node is assigned a \emph{divergence cost} $C_{i}^{div}(\ur_{i})$, i.e.,
 \begin{align}
 \begin{split} \label{prob.OFP_Basic}
 \min_{\ub, \mub} \quad& \sum_{i=1}^{|\Nodeset|}C_{i}^{div}(\ur_{i}) + \sum_{k=1}^{|\Edgeset|}C_{k}^{flux}(\mur_{k}) \\
 \mbox{s.t.\quad} & \ub + E\mub = 0.
 \end{split}
 \end{align}

 The problem \eqref{prob.OFP_Basic} admits a dual problem with a very characteristic structure.
To form the dual problem, one can replace the divergence $\ur_{i}$ and flow $\mur_{k}$ variables in the objective functions with artificial variables $\tilde{\ur}_{i}$ and $\tilde{\mur}_{k}$, respectively, and introduce the artificial constraints $\ur_{i} = \tilde{\ur}_{i}$, $\mur_{k} = \tilde{\mur}_{k}$.  These artificial constraints can be dualized with Lagrange multipliers $\yr_{i}$ and $\zetar_{k}$, respectively. 
The objective functions of the dual problem turn out to be the convex conjugates of the original cost functions, i.e., 
$$
C_{i}^{pot}(\yr_{i}) := C_{i}^{div, \star}(\yr_i) = -\inf_{\tilde{\ur}_{i}} \{ C_{i}^{div}(\tilde{\ur}_{i}) - \yr_{i}\tilde{\ur}_{i} \}
$$ 
and
$
C_{k}^{ten}(\zetar_{i}) := C_{k}^{flux, \star}(\zetar_i).
$
In the dual, the linear constraint  $\zetab = E^{\top}\yb$ must hold, such that is becomes an \emph{optimal potential problem}
\begin{align}
\begin{split} \label{prob.OPP_Basic}
\min_{\yb, \zetab} &\quad \sum_{i=1}^{|\Nodeset|} C^{pot}_{i}(\yr_{i}) + \sum_{k=1}^{|\Edgeset|}C_{k}^{ten}(\zetar_{k}), \\ \mbox{s.t.}& \quad \zetab = E^{\top}\yb.
\end{split}
\end{align}
We provide in the sequel an interpretation of cooperative control problems for a certain class of passive systems in the context of the network optimization theory. 

\section{Passivity-based Cooperative Control} \label{sec.Problem} 
The basic model involving networks of passive dynamical systems with diffusive couplings is now introduced. A new notion of passivity, called \emph{maximal equilibrium independent passivity}, is presented and we demonstrate it to be a well-suited concept for cooperative control.
\subsection{A Canonical Dynamic Network Model}
Networks of dynamical systems defined on an undirected graph $\mc{G} = (\Nodeset,  \Edgeset)$ are considered where each node represents a single-input single-output (SISO) system 
\begin{align}\label{sys.Node}
\begin{split}
\Sigma_{i}: \hspace{2em} \dot{x}_{i}(t) &= f_{i}(x_{i}(t),u_{i}(t),\wr_{i}), \\
  y_{i}(t) &= h_{i}(x_{i}(t),u_{i}(t),\wr_{i}), \qquad i \in \Nodeset,
\end{split}
\end{align}
with state $x_{i}(t) \in \mathbb{R}^{p_{i}}$, input $u_{i}(t) \in \mathbb{R}$, output $y_{i}(t) \in \mathbb{R}$ and constant external signal $\wr_{i}$.
In the following, we adopt the notation $\yd(t) = [y_{1}(t), \ldots, y_{|\Nodeset|}(t)]^{\top}$ and 
$\ud(t) = [u_{1}(t), \ldots, u_{|\Nodeset|}(t)]^{\top}$ for the stacked output and input vectors.  Similarly, we use $\xd(t) \in \mathbb{R}^{\sum_{i=1}^{|\Nodeset|}p_{i}}$ for the stacked state vector, $\wb $ for the external signals and write $\dot{\xd} = \fb(\xd,\ud,\wb), \yd = \hb(\xd,\ud,\wb)$ for the complete stacked dynamical system.
\begin{figure}[!t]

\begin{center}
\includegraphics[trim= 8cm 11.5cm 7.5cm 11cm, width=0.3\textwidth]{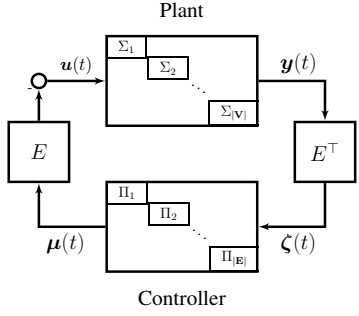}
\caption{Block-diagram of the canonical passivity-based cooperative control structure.}
\label{fig.BlockDiagram}
\end{center}
\end{figure} 
To each edge $k \in \Edgeset,$ connecting two nodes $i,j \in \Nodeset$, we associate the relative output $\zeta_{k}(t) = y_{i}(t) -y_{j}(t)$. The relative outputs can be defined with the incidence matrix as
\begin{align} \label{eqn.Interconnectionzeta}
\zetad(t) = E^{\top}\yd(t).
\end{align}
The relative outputs $\zetad(t)$ drive dynamical systems placed on the edges of $\mc{G}$ that are of the form
\begin{align} \label{sys.Controller1}
\begin{split}
\Pi_{k}: \quad \dot{\z}_{k}(t) &=\zeta_{k}(t),  \\
 \mu_{k}(t) &= \psi_{k}(\z_{k}(t)), \quad k \in \Edgeset.
 \end{split}
\end{align}
The nonlinear output functions  $\psi_{k}$ will be specified later on. The systems \eqref{sys.Controller1} will in the following be called controllers.
The output of a controller $\Pi_{k}$ influences the two incident systems as
\begin{align} \label{eqn.Interconnectionu}
\ud(t) = -E\mud(t).
\end{align}

The complete dynamical network \eqref{sys.Node}, \eqref{eqn.Interconnectionzeta},  \eqref{sys.Controller1}, \eqref{eqn.Interconnectionu} is illustrated in Figure \ref{fig.BlockDiagram}.

\begin{remark}
The model  \eqref{sys.Node}, \eqref{eqn.Interconnectionzeta},  \eqref{sys.Controller1}, \eqref{eqn.Interconnectionu}  includes the class of \emph{diffusively coupled networks} of the form
\begin{align*}
\ddot{\chi}_{i} =  f_{i}(\dot{\chi}_{i}) + \wr_{i} + \sum_{j \in \mathcal{N}_{i}}\psi_{ij}(\chi_{j}-\chi_{i}),
\end{align*}
where $\chi_{i} \in \mathbb{R}$ is a dynamical state and $\mathcal{N}_{i}$ is the set of neighbors of node $i$ in $\mc{G}$. If the nonlinear diffusive couplings $\psi_{ij} (\chi_{j}-\chi_{i})$ are realized by odd functions and $\psi_{ij} = \psi_{ji}$, then the system can be represented in the form \eqref{sys.Node}, \eqref{eqn.Interconnectionzeta},  \eqref{sys.Controller1}, \eqref{eqn.Interconnectionu}, with $x_{i} = \dot{\chi}$. This model has evolved as a standard model, e.g., studied in \cite{Arcak2007}, where passivity was identified as a central concept for convergence analysis.
\end{remark}

\begin{remark}
The model \eqref{sys.Node}, \eqref{eqn.Interconnectionzeta},  \eqref{sys.Controller1}, \eqref{eqn.Interconnectionu} is also closely related  \emph{Hamiltonian systems on graphs}, as studied in \cite{Schaft2012}. Suppose there exists a Hamiltonian function
 $H : \mathbb{R}^{|\Edgeset|} \times \mathbb{R}^{|\Nodeset|} \rightarrow \mathbb{R}$, then a port-Hamiltonian system on a graph takes the form
\begin{align}
\begin{bmatrix}
\dot{\zd}(t) \\ \dot{\xd} (t)
\end{bmatrix} = 
\begin{bmatrix}
0 & E^{\top} \\
-E & -D
\end{bmatrix}
\begin{bmatrix}
\frac{\partial H}{\partial \zd}(\zd(t),\xd(t)) \\
\frac{\partial H}{\partial \xd}(\zd(t),\xd(t))
\end{bmatrix}
+ 
\begin{bmatrix}
0 \\
G
\end{bmatrix} \wr.
\end{align} 
 The matrix $D$ is a positive semidefinite ``damping'' matrix. If $D$ is a diagonal matrix, and $\frac{\partial H}{\partial \zd}(\zd(t),\xd(t))$  and  $\frac{\partial H}{\partial \xd}(\zd(t),\xd(t))$ are solely functions of $\zd(t)$ and $\xd(t)$, respectively, then the model is  in the form \eqref{sys.Node}, \eqref{eqn.Interconnectionzeta},  \eqref{sys.Controller1}, \eqref{eqn.Interconnectionu}.\footnote{It is pointed out in \cite{Schaft2012} that the Hamiltonian function normally splits into a sum of functions on the edges and on the nodes. In this case, the model corresponds to \eqref{sys.Node}, \eqref{eqn.Interconnectionzeta},  \eqref{sys.Controller1}, \eqref{eqn.Interconnectionu}.}
\end{remark}

\subsection{Passivity as a Sufficient Condition for Convergence}
A common theme in the existing literature is to exploit passivity properties of the dynamical systems for a convergence analysis. The convergence results can be traced back to well-known feedback theorems \cite{Khalil2002}, and we review a basic convergence result here.
From here on we use the notational convention that italic letters denote dynamic variables, e.g., $y(t)$, and letters in normal font denote constant signals, e.g., $\yr$.

\begin{assumption} \label{ass.Passivity}
 There exist constant signals $\ub$, $\yb$, $\mub$, $\zetab$ such that $\ub = - E \mub$, $\zetab = E^{\top}\yb$ and
\begin{enumerate}[i)]
 \item  each dynamic system \eqref{sys.Node} is output strictly passive with respect to $\ur_{i}$ and $\yr_{i}$, i.e., there exists a positive semi-definite storage function $S_{i}(x_{i}(t))$ and a constant $\rho_{i} > 0$ such that
\begin{align} \label{eqn.Passivity1}
 \dot{S}_{i} \leq - \rho_{i}\|y_{i}(t) - \yr_{i}\|^{2} + (y_{i}(t) - \yr_{i})(u_{i}(t) - \ur_{i});
\end{align}
 \item each controller  \eqref{sys.Controller1} is passive with respect to $\zetar_{k}$ and $\mur_{k}$, i.e., there exists a positive semi-definite storage function $W_{k}(\z_{k}(t))$ such that
\[
 \dot{W}_{k} \leq (\mu_{k}(t) - \mur_{k})(\zeta_{k}(t) - \zetar_{k}).
\]
\end{enumerate}
\end{assumption}

Now, the basic convergence result follows directly.
\begin{theorem}[{\small Convergence of Passive Networks}] \label{thm.BasicConvergence} Consider the dynamical network \eqref{sys.Node}, \eqref{eqn.Interconnectionzeta}, \eqref{sys.Controller1}, \eqref{eqn.Interconnectionu} and suppose Assumption \ref{ass.Passivity} holds, then the output variables $\yd(t)$ converge to a constant steady-state value $\yb$, i.e., $\lim_{t \rightarrow \infty} \yd(t) \rightarrow \yb$.
\end{theorem}
\begin{proof}
The passivity condition implies that
\begin{align*}
\sum_{i=1}^{|\Nodeset|} \dot{S}_{i} &\leq   -\sum_{i=1}^{|\Nodeset|}\rho_{i}\|y_{i}(t) - \yr_{i}\|^{2} + (\yd(t) - \yb)^{\top}(\ud(t) - \ub) \\
&=  -\sum_{i=1}^{|\Nodeset|}\rho_{i}\|y_{i}(t) - \yr_{i}\|^{2} - (\zetad(t) - \zetab)^{\top}(\mud(t) - \mub) \\
&\leq  -\sum_{i=1}^{|\Nodeset|}\rho_{i}\|y_{i}(t) - \yr_{i}\|^{2} - \sum_{k=1}^{|\Edgeset|}\dot{W}_{k}.
\end{align*} 
One can bring  $\sum_{k=1}^{|\Edgeset|}\dot{W}_{k}$ to the left of the inequality and invoking Barbalat's lemma \cite{Khalil2002} to conclude convergence, i.e.,  $\lim_{t \rightarrow \infty} \| \yd(t) - \yb\| \rightarrow 0$.
\end{proof}

The appeal of this convergence result is that it decouples the dynamical systems layer and the network layer. Only the input-output behavior must be shown to be passive to conclude convergence of the overall network.

\subsection{Equilibrium Independent Passivity}

A critical aspect of the previous result relates to the assumption on the existence of the constant signals $\ub$, $\yb$, $\mub$, $\zetab$  that satisfy Assumption \ref{ass.Passivity}. The equilibrium configuration depends on the properties of all systems in the network and the desired passivity property cannot be verified locally. 
To overcome this issue, the concept of \emph{equilibrium independent passivity} was introduced  in \cite{Hines2011}. Equilibrium independent passivity requires a system to be passive independent of the equilibrium point to which it is regulated. 

\begin{definition}[\cite{Hines2011}] The system \eqref{sys.Node} is said to be \emph{(output strictly) equilibrium independent passive} if there exists a set $\mc{U}_{i} \subset \mathbb{R}$ and a continuous function $k_{x,i}(\ur)$, defined on $\mc{U}_{i}$, such that
 i) for any constant signal $\ur_{i} \in \mc{U}_{i}$  the constant signal $\xr_{i} = k_{x,i}(\ur_{i})$ is an equilibrium point of \eqref{sys.Node}, i.e., $0 = f_{i}(\xr_{i},\ur_{i},\wr_{i})$, and 
 ii) the system is passive with respect to $\ur_{i}$ and $\yr_{i} = h_{i}(k_{x,i}(\ur),\ur_{i},\wr_{i})$; that is, for each $\ur_{i} \in \mc{U}_{i}$ there exists a storage function such that the inequality \eqref{eqn.Passivity1} holds (with $\rho_{i} \geq 0$ for equilibrium independent passivity and $\rho_{i} > 0$ for output-strictly equilibrium independent passivity). 
\end{definition}

The relevance of equilibrium independent passivity for the analysis of dynamical networks can be readily seen. 
If the systems \eqref{sys.Node} and \eqref{sys.Controller1} are output-strictly equilibrium independent passive and equilibrium independent passive, respectively, one has to verify only that an equilibrium trajectory exists in the respective sets to make the basic convergence proof of Theorem \ref{thm.BasicConvergence} applicable. The exact equilibrium point need not be known.

One important implication of equilibrium independent passivity is that the equilibrium input-output map must be \emph{monotone}, and even co-coercive, if the system is output-strictly equilibrium independent passive, see \cite{Hines2011}.

\subsection{Maximal Equilibrium Independent Passivity}

While equilibrium independent passivity turns out to be an useful concept for network analysis, the given definition excludes some important systems.
Consider for example a simple integrator, i.e.,
$ \dot{x}_{i}(t) = u_{i}(t),\; y_{i}(t) = x_{i}(t) $. 
It is well known that the integrator is passive with respect to $\mc{U}_{i} = \{0\}$ and any output value $\yr _{i} \in \mathbb{R}$.\footnote{Passivity with respect to an arbitrary output $\yr_i \in \mathbb{R}$ can be readily see with the storage function $S_{i}(x_{i}(t)) = \frac{1}{2}(x_{i}(t) - \yr_{i})^{2}$.}
However, the equilibrium input-output map is not a (single-valued) \emph{function} such that the integrator is not equilibrium independent passive, as defined in \cite{Hines2011}. 

Motivated by this example, we propose here a refinement of equilibrium independent passivity. In particular, we do not require the equilibrium input-output maps $\kyi{i}$ to be \emph{functions}, but instead allow them to be \emph{relations} (or curves in $\mathbb{R}^{2}$). 
That is, $\kyi{i}$ is the set of all pairs $(\ur_{i},\yr_{i}) \in \mathbb{R}^{2}$ that are equilibrium input-output relations.  The domain of the relation is the set $\mc{U}_i$, i.e., $\mbox{dom}\,\kyi{i} := \mc{U}_i$. We will sometimes write $\kyi{i}(\ur_{i})$ to denote the set of all $\yr_{i}$ such that $(\ur_{i},\yr_{i}) \in \kyi{i}$. This gives an interpretation of $\kyi{i}(\ur_{i})$ as \emph{set-valued map}.
For the integrator example described above, the equilibrium input output relation is the vertical line through the origin, i.e., $\kyi{i} = \{ (\ur_{i},\yr_{i}) : \ur_{i} = 0, \yr_{i} \in \mathbb{R}\}$.
For relations in $\mathbb{R}^{2}$ we review the concept of \emph{maximal monotonicity}. 

\begin{definition}[\cite{Rockafellar1998}]\label{def.maxmonotone}
A relation $\kyi{i}$ is said to be \emph{maximally monotone} if it cannot be embedded into a larger monotone relation.  Equivalently, the relation $\kyi{i}$ is a maximal monotone relation if and only if
\begin{enumerate}[i)]
\item for arbitrary $(\ur_{i},\yr_{i}) \in \kyi{i}$ and $(\ur_{i}',\yr_{i}') \in \kyi{i}$ one has either $\ur_{i} \leq \ur_{i}'$ and $\yr_{i} \leq \yr_{i}'$, denoted by $(\ur_{i},\yr_{i}) \leq (\ur_{i}',\yr_{i}')$, or $(\ur_{i},\yr_{i}) \geq (\ur_{i}',\yr_{i}')$, and
\item  for arbitrary $(\ur_{i},\yr_{i}) \notin \kyi{i}$ there exists $(\ur_{i}',\yr_{i}') \in \kyi{i}$ such that neither $(\ur_{i},\yr_{i}) \leq (\ur_{i}',\yr_{i}')$ nor $(\ur_{i},\yr_{i}) \geq (\ur_{i}',\yr_{i}')$. 
\end{enumerate}
\end{definition}
We refer to \cite{Rockafellar1998} for a detailed treatment of maximal monotone relations. It is not difficult to see that the equilibrium input-output relation of the integrator system discussed above is maximally monotone.
Based on this definition, a refined version of equilibrium independent passivity can be introduced. 
Please note that only SISO systems are considered in this paper and therefore the following definition applies only to SISO systems.

\begin{definition}[{\small Maximal equilibrium independent passivity}]
A dynamical SISO system \eqref{sys.Node} is said to be maximal equilibrium independent passive if there exists a maximal monotone relation $\kyi{i} \subset \mathbb{R}^{2}$ such that for all $(\ur_{i},\yr_{i}) \in \kyi{i}$ there exits a positive semi-definite storage function $S_{i}(x_{i}(t))$ satisfying
\begin{align} \label{eqn.Passivity2}
 \dot{S}_{i} \leq  (y_{i}(t) - \yr_{i})(u_{i}(t) - \ur_{i}).
\end{align}

Furthermore, it is \emph{output-strictly maximal equilibrium independent passive} if additionally there is a constant $\rho_{i} > 0$ such that
\begin{align} \label{eqn.Passivity3}
 \dot{S}_{i} \leq - \rho_{i}\|y_{i}(t) - \yr_{i}\|^{2} + (y_{i}(t) - \yr_{i})(u_{i}(t) - \ur_{i}).
\end{align}
\end{definition}

The new notion of maximal equilibrium independent passivity is closely related to the definition of \cite{Hines2011}. 
In fact, any equilibrium independent system with $\mc{U}_{i} = \mathbb{R}$ is also maximal equilibrium independent passive. 
This includes in particular affine dynamical systems
\begin{align}
\begin{split} \label{sys.LTI}
\dot{x}(t) &= Ax(t) + Bu(t) + P\wr\\
y(t) &= Cx(t) + Du(t) + G\wr,
\end{split}
\end{align}
that were shown in \cite{Hines2011} to be output strictly equilibrium independent passive if they are output-strictly passive in the classical
sense for $\wr = 0$ and if $A$ is invertible.
The equilibrium input-output relation is then the (single-valued) affine function (and thus a maximal monotone relation)
$
k_{y}(\ur) = \left( -CA^{-1}B + D \right)\ur + \left(-CA^{-1}P + G\right)\wr.
$
Note that this is the dc-gain of the linear system plus the constant value determined by the exogenous inputs.

The two definitions also both include scalar nonlinear systems of the form
\begin{align} \label{sys.Scalar}
\dot{x}(t) = -f(x(t)) + u(t), \; y(t) = x(t), 
\end{align}
with $x(t) \in \mathbb{R}, u(t) \in \mathbb{R}, y(t)\in \mathbb{R}$, for which 
$
(x'(t)-x''(t))\bigl(f(x'(t)) - f(x''(t)) \bigr)   \geq  \gamma (x'(t)-x''(t))^{2}
$
for all $x', x'' \in \mathbb{R}$. 

However, the integrator is the central example of a system that is included in the new definition of maximal equilibrium independent passivity, but not in the original definition of \cite{Hines2011}.

In the following section, networks of the structure  \eqref{sys.Node}, \eqref{eqn.Interconnectionzeta},  \eqref{sys.Controller1}, \eqref{eqn.Interconnectionu} consisting of maximal equilibrium independent passive systems will be considered. It will be shown that these networks admit a certain \emph{inverse optimality}, in the sense that they converge to the solutions of several dual pairs of network optimization problems of the form \eqref{prob.OFP_Basic} and \eqref{prob.OPP_Basic}.
This result establishes a connection between the new definition of passivity and convex network optimization theory.

\section{Output Agreement Analysis} \label{sec.OutputAgreement}

We now investigate the steady-state behavior of the dynamical network  \eqref{sys.Node}, \eqref{eqn.Interconnectionzeta}, \eqref{sys.Controller1}, \eqref{eqn.Interconnectionu} and characterize an associated \emph{inverse optimality} for these systems. 
To prepare the following discussion, we introduce some additional notation. We will write $\kyf(\ub)$ for the stacked input-output relations, that is $\yb \in \kyf(\ub)$ means $\yr_{i} \in \kyi{i}(\ur_{i})$ for all $i \in \Nodeset$. Similarly we will write $\mc{U} = \mc{U}_{1} \times \cdots \times \mc{U}_{|\Nodeset|}$ and $\mc{Y} = \mc{Y}_{1} \times \cdots \times \mc{Y}_{|\Nodeset|}$ to indicate the domain and range of  $\kyf(\ub)$.

\subsection{The Plant Level}
The first observation we make is that a steady-state of the network  \eqref{sys.Node}, \eqref{eqn.Interconnectionzeta},  \eqref{sys.Controller1}, \eqref{eqn.Interconnectionu} requires all systems to be in \emph{output agreement}.
Suppose that $\xb$ and $\zb$ are steady-state solutions of the network \eqref{sys.Node}, \eqref{eqn.Interconnectionzeta},  \eqref{sys.Controller1}, \eqref{eqn.Interconnectionu}, and let $\yb$ be the corresponding steady-state output, then 
\[ \yb  = \beta \1, \]
for some $\beta \in \mathbb{R}$, called the \emph{agreement value}.
Output agreement follows from the steady-state condition $\dot{\zd} = 0$, that requires $\yb \in \mc{N}(E^{\top})$. As $\mc{G}$ is connected, $\yb \in \mc{N}(E^{\top})$ is equivalent to $\yb  = \beta \1$ for some $\beta$. 

The existence of an output agreement solution depends on properties of the nodes \eqref{sys.Node} and the network topology. In particular, the existence of an output agreement solution is related to the network equilibrium feasibility problem:
\begin{align}
\begin{split} \label{prob.EquilibriumProblem}
\mbox{Find\;} &\ub \in \mc{R}(E), \;  \yb \in\mc{N}(E^{\top})\\
\mbox{such that }&\yb \in \kyf(\ub).
\end{split}
\end{align}
A necessary condition for the existence of an output agreement solution is now the following.
\begin{lemma}[{\small Necessary Condition}]
 If the network \eqref{sys.Node}, \eqref{eqn.Interconnectionzeta},  \eqref{sys.Controller1}, \eqref{eqn.Interconnectionu} has a steady-state solution $\ub$, $\yb$, then this steady-state is a solution to \eqref{prob.EquilibriumProblem}.
\end{lemma} 

\begin{proof}
The  steady state condition for the plant and for the controller require  $\yb \in  \kyf(\ub)$ and $\yb \in \mc{N}(E^{\top})$, respectively. Additionally, the interconnection \eqref{eqn.Interconnectionu} implies that  $\ud(t) \in \mc{R}(E)$, and consequently that $\ub  \in \mc{R}(E)$.
\end{proof}
%
To obtain further insights into the properties of an output agreement solution, we will next establish a connection to network optimization problems and show that certain duality relations hold. 
Therefore, some results relating maximal monotone relations and convex functions are recalled from \cite{Rockafellar1998}. 
A first observation is that one can extend any maximal monotone relation $\kyi{i} \subset \mathbb{R}^{2}$ with domain $\mc{U}_{i}$ to a maximal monotone relation on $\mathbb{R}$ by setting it to $-\infty$ for all $\ur_{i}$ `left' of $\mc{U}_{i}$ and $+\infty$ for all $\ur_{i}$ `right' of $\mc{U}_{i}$.\footnote{ Note that since $\kyi{i}$ is a maximal monotone relation, $\mc{U}_i$ is a connected interval on $\mathbb{R}$} 
Now, we recall the following result of \cite[Thm. 24.9]{Rockafellar1997} that holds for $\mathbb{R}$: 
\begin{theorem}[\cite{Rockafellar1997}]
{The subdifferential for the closed proper convex functions on $\mathbb{R}$ are the maximal monotone relations from $\mathbb{R}$ to $\mathbb{R}$. }
\end{theorem}
Thus, one can associate to any maximal monotone relation, and consequently to any maximal equilibrium independent passive system, a closed proper convex function $K_{i}: \mathbb{R} \rightarrow \mathbb{R}$ that is unique up to an additive constant,
 such that
\begin{align}\label{eqn.IOMap}
 \partial K_{i}(\ur_{i}) = \kyi{i}(\ur_{i}) \quad \forall \ur_{i} \in \mc{U}_{i}.
\end{align} 
 If $\mc{U}_{i}$ is not the complete $\mathbb{R}$ and the maximal monotone relation has been extended as described above, then $K_{i}(\ur_{i}) = +\infty$ for all $\ur_{i} \notin \mc{U}_{i}$.
If the equilibrium input-output relation is a continuous single-valued function from $\mathbb{R}$ to $\mathbb{R}$ then $K_{i}(\ur_{i})$ is differentiable and 
$
 \nabla K_{i}(\ur_{i}) = \kyi{i}(\ur_{i}).
$ 

We will call $K_{i}(\ur_{i})$ the \emph{cost function} of the maximal equilibrium independent passive system $i$. 
Its convex conjugate, defined as in \eqref{eqn.Conjugate}, i.e.,
$K^{\star}_{i}(\yr_{i}) = \sup_{\ur_{i}} \; \{ \yr_{i}\ur_{i} - K_{i}(\ur_{i})\},$
is called the \emph{potential function} of  system $i$.  

The steady-states of the dynamical network of maximal equilibrium independent passive systems are intimately related to the following pair of dual network optimization problems.


 \textbf{Optimal Flow Problem:} Consider the the following \emph{optimal flow problem}
\begin{align}\tag{OFP1}
\begin{split}  \label{prob.OFP_Unconstrained}
\min_{\ub, \mub} \quad &\sum_{i=1}^{|\Nodeset|}K_{i}(\ur_{i}) \\
\mbox{s.t.}\quad &\ub + E\mub = 0.
\end{split}
\end{align}
This problem is of the form of an optimal flow problem \eqref{prob.OFP_Basic}.
The cost on the divergences  (in/out-flows) $\ub \in \mathbb{R}^{|\Nodeset|}$ are the integral functions of the equilibrium input-output relations, i.e., $C_{i}^{div} = K_{i}$, and the flows  $\mub \in \mathbb{R}^{|\Edgeset|}$ on the edges are not penalized, i.e., $C_{k}^{flux} = 0$. 

 \textbf{Optimal Potential Problem:} Dual to the optimal flow problem, we define the following \emph{optimal potential problem} 
\begin{align} \tag{OPP1} \label{prob.PotentialUnconstrained2}
\begin{split}
\min_{\yr_{i}} \; &\sum_{i=1}^{|\Nodeset|} K_{i}^{\star}(\yr_{i}), \\
 \mbox{s.t.\quad}&  E^{\top}\yb = 0.
 \end{split}
\end{align} 
This problem is in the form \eqref{prob.OPP_Basic}. The convex conjugates of the integral functions of the equilibrium input-to-output maps are the costs for the potential variables $\yb \in \mathbb{R}^{|\Nodeset|}$ of the nodes, i.e., $C_{i}^{pot} =  K_{i}^{\star}$.
The constraint $E^{\top}\yb=0$ enforces a balancing of the potentials over the complete network. The problem can be written in the standard form  \eqref{prob.OPP_Basic}, by choosing $C^{ten}_{k} = I_{0}$, i.e., the indicator function for the point zero.  
To simplify the presentation, we will use the short-hand notation $\mathbf{K}(\ub) := \sum_{i=1}^{|\Nodeset|}K_{i}(\ur_{i})$ and $\mathbf{K}^{\star}(\yb) := \sum_{i=1}^{|\Nodeset|} K_{i}^{\star}(\yr_{i})$.

The main result of this paper is that the the output agreement steady-states in a network of maximal equilibrium independent passive systems admit an inverse optimality.
\begin{theorem}[{\small Inverse Optimality of Output Agreement}]\label{thm.Necessary}
Suppose all node dynamics \eqref{sys.Node} are maximal equilibrium independent passive. If the network \eqref{sys.Node}, \eqref{eqn.Interconnectionzeta},  \eqref{sys.Controller1}, \eqref{eqn.Interconnectionu} has a steady-state solution $\ub$, $\yb$, then 
 (i) $\ub$ is an optimal solution to \eqref{prob.OFP_Unconstrained},
 (ii) $\yb$ is an optimal solution to \eqref{prob.PotentialUnconstrained2}, and
(iii) in the steady-state \eqref{prob.OFP_Unconstrained} and \eqref{prob.PotentialUnconstrained2} have same value with negative sign, i.e., 
$\sum_{i=1}^{|\Nodeset|} K_{i}(\ur_{i}) +  \sum_{i=1}^{|\Nodeset|} K_{i}^{\star}(\yr_{i}) = 0$.
 \end{theorem}

\begin{proof}
It is sufficient to show that the conclusions hold if the equilibrium problem \eqref{prob.EquilibriumProblem} has a solution.
If  there is a solution $\ub$, $\yb$ to \eqref{prob.EquilibriumProblem}, then $\ub \in \mathcal{R}(E) \cap \mc{U}$, while $\yb \in \mathcal{N}(E^{\top}) \cup \mc{Y}$. Thus, both optimization problem have a feasible solution and are finite.
Consider now the Lagrangian function of \eqref{prob.OFP_Unconstrained} with multiplier $\tilde{\yb}$, i.e.,
\[ \mc{L}(\ub,\mub,\tilde{\yb}) = \sum_{i=1}^{|\Nodeset|}K_{i}(\ub) - \tilde{\yb}^{\top}\ub + \tilde{\yb}^{\top}E\mub. \]
For $\ub$ to be the a solution to \eqref{prob.OFP_Unconstrained}, it is necessary and sufficient that 
\begin{align}
\tilde{\yb} \in \partial \boldsymbol{K}(\ub) 
\end{align}
for the optimal multiplier $\tilde{\yb}$. Thus, since  $\partial \boldsymbol{K}(\ub) = \kyf(\ub)$, the multiplier satisfies $\tilde{\yb} \in  \kyf(\ub)$. 

To conclude that $\ub$ is an optimal solution, it remains to show that the equilibrium trajectory $\yb$ is an optimal multiplier, i.e., $\yb = \tilde{\yb}$. As $\yb$ satisfies the equilibrium condition, it only remains to show that $\tilde{\yb} = \mc{N}(E^{\top})$.
Let $s(\tilde{\yb}) = \inf_{\ub,\mub} \mc{L}(\ub,\mub,\tilde{\yb})$. Now, if $\tilde{\yr} \notin \mc{N}(E^{\top})$ then $s(\tilde{\yb})$ is unbounded below.
For $\tilde{\yb} \in \mc{N}(E^{\top})$ it follows that $s(\tilde{\yr}) = - \mathbf{K}^{\star}(\tilde{\yb})$. Thus, the supremum problem is identical to \eqref{prob.PotentialUnconstrained2} with the negative objective function and both problems will have the same solution. 
Now, if the network equilibrium problem has a solution, than there must exists $\ub$ and $\yb$ satisfying the optimality conditions for the dual pair of optimization problems \eqref{prob.OFP_Unconstrained} and \eqref{prob.PotentialUnconstrained2}. 
Finally, as the steady-state solution is an optimal to both problems \eqref{prob.OFP_Unconstrained} and \eqref{prob.PotentialUnconstrained2}, it must be a saddle-point for the Lagrangian function, i.e., it must hold that
\begin{align}\label{eqn.SaddleEquality}
\sup_{\yb} \inf_{\ub,\mub} \;    \mc{L}(\ub,\mub,\yb) =  \inf_{\ub,\mub} \sup_{\yb} \;    \mc{L}(\ub,\mub,\yb). 
\end{align}
Let now $r(\ub,\mub) = \sup_{\yb} \;    \mc{L}(\ub,\mub,\yb)$. It follows that $r(\ub,\mub) = \mathbf{K}(\ub)$ if $\ub + E\mub = 0$ and $r(\ub,\mub) = + \infty$ otherwise.
Additionally, we have already seen that $s(\yr) = \inf_{\ub,\mub} \;    \mc{L}(\ub,\mub,\yb)$ is $s(\yr) = - \mathbf{K}^{\star}(\yb)$ if $\yb \in \mc{N}(E^{\top})$ and $s(\yr) = - \infty$ otherwise.
For \eqref{eqn.SaddleEquality} to hold, the optimal solution $\ub \in \mc{R}(E)$ and $\yb \in \mc{N}(E^{\top})$ must be such that  $\boldsymbol{K}(\ub) + \boldsymbol{K}^{\star}(\yb) = 0.$ As shown before, the steady-states of the dynamic network are optimal solutions to \eqref{prob.OFP_Unconstrained} and \eqref{prob.PotentialUnconstrained2} and must therefore satisfy the previous equality.
\end{proof}

The connection between the necessary condition for the existence of an agreement steady-state of the dynamical network and the dual pair of network optimization problems opens the way to use well-known tools form convex analysis for investigating the properties of output agreement steady-states in dynamic networks.

\begin{corollary}[{\small Existence}]
Suppose all node dynamics are maximal equilibrium independent passive with $\mc{U}_{i} = \mathbb{R}$ and $\mc{Y}_{i} = \mathbb{R}$, then an output agreement steady-state exists.
\end{corollary}
\begin{proof}
Under the given assumption both optimization problems have a feasible solution and strong duality holds. The optimal primal-dual solution pair solves the equilibrium problem \eqref{prob.EquilibriumProblem} and thus corresponds to a possible output agreement steady state.
\end{proof}

\begin{corollary}[{\small Uniqueness}]\label{cor.Uniqueness}
If the dynamical systems \eqref{sys.Node} are maximal equilibrium independent passive with a nonempty $\mc{U}_{i}$ and a strongly monotone equilibrium input-output function $\kyi{i}$ satisfying $\lim_{\ell \rightarrow \infty} |\kyi{i}(\ur^{\ell})| \rightarrow \infty$ whenever $\ur^{1}, \ur^{2}, ...$ is a sequence in $\mc{U}_{i}$ converging to a boundary point of $\mc{U}_{i}$, then there exists at most one pair $(\ub, \yb)$ that can be a steady-state solution. 
\end{corollary}
\begin{proof}
From the assumptions follow that $K_{i}(\ur_{i})$ are differentiable and essentially smooth convex functions (see \cite[p. 251]{Rockafellar1997}). Thus, \eqref{prob.OFP_Unconstrained} can have at most one solution. 
If such a solution exists, then the dual problem also has a solution.  
\end{proof}
\begin{corollary}[{\small Agreement Value}]
Assume the same assumptions as for Corollary \ref{cor.Uniqueness} hold. If an output agreement steady-state exists, the agreement value $\beta$ satisfies
\begin{align} \label{eqn.OptimalityAgreement}
\sum_{i=1}^{|\Nodeset|}\kyi{i}^{-1}(\beta) = 0.
\end{align}
\end{corollary}
\begin{proof}
It follows from Theorem 26.1 in \cite{Rockafellar1997} that $\nabla K_{i}^{\star}(\yr_{i}) = \kyi{i}^{-1}(\yr_{i})$.
Thus, after replacing $\yb$ in \eqref{prob.PotentialUnconstrained2} with $\yb = \beta \1$, the optimality condition of \eqref{prob.PotentialUnconstrained2} corresponds exactly to \eqref{eqn.OptimalityAgreement}. 
\end{proof}
\begin{remark}
The above results apply to networks of homogeneous or heterogeneous maximally equilibrium independent systems.  For homogeneous systems, or more generally, for systems where the intersection of the equilibrium input-output maps $\kyi{i}$ is a single point, the solution to \eqref{prob.PotentialUnconstrained2} is simply that intersection point.  This set-up is considered in various passivity-based cooperative control approaches such as \cite{Arcak2007}, \cite{Chopra2006}. 
Thus, it is precisely the \emph{heterogeneous} case, i.e., when the equilibrium input-output maps do not all intersect at the same point,\footnote{This means that there are at least two systems that have distinct equilibria if they are not coupled.} where the presented analysis methods give new insights and turn out to be a powerful network analysis tool. 

\end{remark}

\subsection{The Control Level}
It remains to investigate when the controller dynamics \eqref{sys.Controller1}  can realize an output agreement steady-state. In particular, in the steady-state configuration, the controller \eqref{sys.Controller1} must generate a signal $\mub$ that corresponds to the desired control input. Suppose a solution $\ub$ to \eqref{prob.EquilibriumProblem} is known, then the controller must be such that the following static network equilibrium feasibility problem has a solution:
\begin{align} \label{prob.Equilibrium2}
\begin{split}
\mbox{Find} \quad &\zb \in \mc{R}(E^{\top}) \\
\mbox{such that }&\ub = - E \bfpsi(\zb).
\end{split}
\end{align}

\begin{lemma}[{\small Necessary and Sufficient Condition}]
The network \eqref{sys.Node}, \eqref{eqn.Interconnectionzeta},  \eqref{sys.Controller1}, \eqref{eqn.Interconnectionu} has a steady-state solution if and only if there exists a solution to \eqref{prob.EquilibriumProblem} and \eqref{prob.Equilibrium2}.
\end{lemma}
\begin{proof}
If the equilibrium problems have a solution $\ub, \yb, \zb$, then $\ub, \yb, \mub = \bfpsi(\zb)$ and $\zetab=0$
are a steady-state solution to \eqref{sys.Node}, \eqref{eqn.Interconnectionzeta},  \eqref{sys.Controller1}, \eqref{eqn.Interconnectionu}.
Any steady-state solution $\ub, \yb, \mub,\zetab$ of \eqref{sys.Node}, \eqref{eqn.Interconnectionzeta},  \eqref{sys.Controller1}, \eqref{eqn.Interconnectionu} solves the two equilibrium problems with $\mub = \bfpsi(\zb)$.  
\end{proof}
Please note that the two equilibrium problems \eqref{prob.EquilibriumProblem} and \eqref{prob.Equilibrium2} are not independent.
However, if \eqref{prob.EquilibriumProblem} has a unique solution, \eqref{prob.Equilibrium2} has no influence on the solution of \eqref{prob.EquilibriumProblem}.

As the required steady-state input $\ub$ is in general not known for the controller design, it seems appropriate to design the controller such that \eqref{prob.Equilibrium2} is feasible for any $\ub \in \mc{R}(E)$.
Again, it will turn out that the feasibility of the network equilibrium problem is intimately related to maximal monotonicity. 
In particular, we show that \eqref{prob.Equilibrium2} has a solution for all $\ub \in \mc{R}(E)$ if $\psi_{k}$ are \emph{strongly monotone functions}.

Following this observation, we now assume that all $\psi_{k}$ are strongly monotone functions. Then, one can associate to each edge $k \in \Edgeset$ a closed, proper strongly convex function $P_{k}: \mathbb{R} \rightarrow \mathbb{R}$ such that
\begin{align}
\nabla P_{k}(\zr_{k}) = \psi_{k}(\zr_{k}).
\end{align}
\begin{lemma}
Suppose the functions $\psi_{k}$ are strongly monotone, than the controller dynamics \eqref{sys.Controller1} is maximal equilibrium independent passive.
\end{lemma}
\begin{proof}
The equilibrium input set for the controller dynamics is solely $\zeta_{k} = 0$. However, the dynamics \eqref{sys.Controller1} is passive with respect to the input $\zeta_{k} = 0$ and any output $\mur_{k} \in \mathbb{R}$. To see this, consider the storage function
\begin{align*}
W_{k}(\z_{k}(t),\zr_{k}) = P_{k}(\z_{k}(t)) - P_{k}(\zr_{k}) - \nabla P_{k}(\zr_{k})(\z_{k}(t) - \zr_{k}),
\end{align*}
where $\zr_{k}$ is such that $\mur_{k} = \nabla P_{k}(\zr_{k})$. From strict convexity of $P_{k}$ follows directly that $W_{k}$ is a positive definite function.\footnote{Note that $W_{k}$ is the Bregman distance associated to $P_{k}$ between $\z_{k}(t)$ and $\zr_{k}$, see \cite{Bregman1967}.} Now, maximal passivity follows immediately from
\[ \dot{W}_{k} = (\nabla P_{k}(\z_{k}(t)) - \nabla P_{k}(\zr_{k}))\zeta_{k}(t) = (\mu_{k}(t) - \mur_{k})(\zeta_{k}(t) - \zetar_{k}),\]
where we used that $\zetar_{k} = 0$.
\end{proof}

It will be shown next that that strong monotonicity of $\psi_{k}$ ensures the existence of an output agreement steady-state solution and that the steady-state solution has additional inverse optimality properties. To see this, consider the following pair of dual network optimization problems.

  \textbf{Optimal Potential Problem:} Let some $\ub = [\ur_{1},\ldots,\ur_{|\Nodeset|}]^{\top} \in \mc{R}(E)$ be given. Consider the following \emph{optimal potential problem}
\begin{align} \tag{OPP2} \label{prob.zExistence}
\begin{split}
\min_{\zb,\vb} \quad &   \sum_{k=1}^{|\Edgeset|}P_{k}(\zr_{k}) + \sum_{i=1}^{|\Nodeset|} \ur_{i} \vr_{i}, \\
 \mbox{s.t.\quad} & \zb =  E^{\top}\vb.
 \end{split}
\end{align}
By its structure, \eqref{prob.zExistence} is an optimal potential problem as defined in \eqref{prob.OPP_Basic}. The potential vector $\vb$ is associated to the linear cost defined by $\ub$, while the tension variables $\zb$ are associated to the integral functions of the coupling nonlinearities. 

 \textbf{Optimal Flow Problem:} The dual problem to \eqref{prob.zExistence} is the following \emph{optimal flow problem}
\begin{align}\tag{OFP2} \label{prob.ODP}
\begin{split}
\min_{\mub} \; &\sum_{k=1}^{|\Edgeset|} P^{\star}_{k}(\mur_{k}) \\
 \mbox{s.t.\;}&  \ub + E\mub = 0, 
 \end{split}
\end{align}
where $P^{\star}_{k}$ is the convex conjugates of $P_{k}$, and $\ub \in \mc{R}(E)$ is a given constant vector. The problem is in compliance with the standard form of optimal flow problems \eqref{prob.OFP_Basic}, as one can introduce artificial divergence variables and add as a cost function the indicator function for the point $\ub$.

\begin{theorem}[{\small Controller Realization}]
Suppose the dynamical network nodes \eqref{sys.Node} are such that the necessary conditions of Theorem \ref{thm.Necessary} are satisfied and the controller dynamics \eqref{sys.Controller1} are such that all $\psi_{k}$ are strongly monotone. Then the network \eqref{sys.Node}, \eqref{eqn.Interconnectionzeta}, \eqref{sys.Controller1}, \eqref{eqn.Interconnectionu} has an output agreement steady-state solution. 
Furthermore, let $\zb$ be the steady-state of the controller in output agreement, then 
(i) $\zb$ is an optimal solution to \eqref{prob.zExistence}, 
(ii) $\mub = \bfpsi(\zb)$ is an optimal solution to \eqref{prob.ODP}, 
(iii) and $\sum_{k=1}^{|\Edgeset|} P^{\star}_{k}(\mur_{k}) + \sum_{k=1}^{|\Edgeset|}  P_{k}(\zr_{k})  = \mub^{\top}\zb$.
\end{theorem}
\begin{proof}
To prove the first claim, it is sufficient to show that for any $\ub \in \mc{R}(E)$ the equilibrium problem \eqref{prob.Equilibrium2} has a solution $\zb$.
At first we note that if $\psi_{k}$ are strongly monotone, then $P_{k}$ are strongly convex and are defined on $\mathbb{R}$. Thus, \eqref{prob.zExistence} has a unique solution for all $\ub \in \mc{R}(E)$. To prove the first claim, it remains to connect the solution of \eqref{prob.zExistence} to the equilibrium condition \eqref{prob.Equilibrium2}.
Any solution $\zb =  E^{\top}\vb$ in \eqref{prob.zExistence} must satisfy the first-order optimality condition
\[  E \nabla \mathbf{P}(E^{\top}\vb) + \ub = 0, \]
where we use the short-hand notation $\mathbf{P} = \sum_{k=1}^{|\Edgeset|}P_{k}$. Since $\nabla \mathbf{P} = \bfpsi$, the optimal solution $\zb =E^{\top}\vb$ to \eqref{prob.zExistence} solves explicitly the equilibrium condition \eqref{prob.Equilibrium2}, proving the first claim.

Now, to prove the remaining statements of the theorem, we consider the Lagrangian of \eqref{prob.zExistence}, i.e.,
\[\mc{L}(\vb,\zb,\tilde{\mub}) =   \sum_{k=1}^{|\Edgeset|}P_{k}(\zr_{k}) + \sum_{i=1}^{|\Nodeset|} \ur_{i} \vr_{i} + \tilde{\mub}^{\top}(-\zb + E^{\top}\vb), \]
with multiplier $\tilde{\mub}$. Define now the dual function as $s(\tilde{\mub}) = \inf_{\vb,\zb} \mc{L}(\vb,\zb,\tilde{\mub})$. Clearly, $s(\tilde{\mub}) = - \infty$ if $ E\tilde{\mub} + \ub\neq 0$, and otherwise $s(\tilde{\mub})  = -\mathbf{P}_{k}^{\star}(\tilde{\mub})$. Thus, the dual problem $\sup s(\tilde{\mu})$ is equivalent to \eqref{prob.ODP} and the dual solution $\tilde{\mub}$ is in fact the optimal solution to \eqref{prob.ODP}. Together with the first order optimality condition this implies that $  \mub = \tilde{\mub} = \nabla \mathbf{P}(\zb) = \bfpsi(\zb)$. 
The last statement, i.e., the strong duality, follows since it must hold that
\[ \sup_{\tilde{\mub}} \inf_{\vb,\zb} \;  \mc{L}(\vb,\zb,\tilde{\mub}) =  \inf_{\vb,\zb} \sup_{\tilde{\mub}} \;  \mc{L}(\vb,\zb,\tilde{\mub}).\]
This implies that $\sup_{\tilde{\mub}} s(\tilde{\mu})$ must take the same optimal value as  \eqref{prob.zExistence}. The statement follows now immediately since $\sup_{\tilde{\mub}} s(\tilde{\mu})$ has the same value as \eqref{prob.ODP} with negative sign, and $\ub^{\top}\vb = - \mub^{\top}\zb$, where $\mub$ is the optimal solution to \eqref{prob.ODP}.
\end{proof}

The two optimization problems provide, on the one hand, explicit statements about the feasibility of the steady-state independent of the required $\ub$, and, on the other hand, additional duality relations. The internal state of the controller \eqref{sys.Controller1}, $\zd(t)$, can be understood as \emph{tensions}, while the output of the controller, $\mud(t)$, can be understood as the corresponding dual \emph{flows}.  
\begin{remark}[{\small Sector Nonlinearities}]
An alternative assumption that is often imposed on the nonlinearities  $\psi_{k}$ in the literature (as, e.g., in \cite{Chopra2006}) is that they are \emph{sector nonlinearities} (see, e.g., \cite[Def. 6.2]{Khalil2002}). The strong monontonicity condition is clearly a stronger assumption.\footnote{After shifting the origin appropriately, a strongly monotone nonlinearity is always a sector nonlinearity.} 
However, the dynamics \eqref{sys.Controller1} with a sector nonlinearity $\psi_{k}$ that is not strongly monotone is not necessarily maximally equilibrium independent passive and it cannot be guaranteed that for any required steady-state input $\ub$ the controller dynamics \eqref{sys.Controller1} is passive with respect to the required steady-state.
Thus, in order to ensure convergence of the network without knowing $\ub$ in advance, the strong monotonicity assumption becomes a necessary requirement.
\end{remark}

\begin{figure*}
\begin{center}
\subfigure[Signals of the Closed-Loop Dynamical System]{\includegraphics[trim= 2cm 16cm 1cm 2.5cm, scale=0.45]{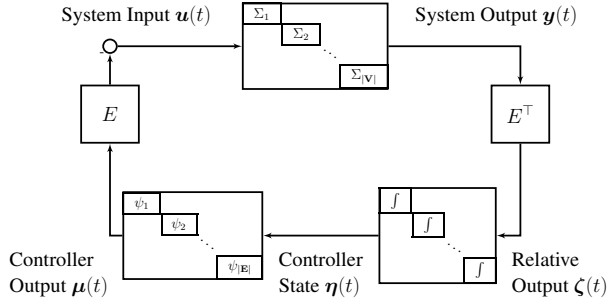}} \hfill
\subfigure[Variables of the Network Theoretic Framework]{\includegraphics[trim= 2cm 16cm 2cm 2.5cm, scale=0.45]{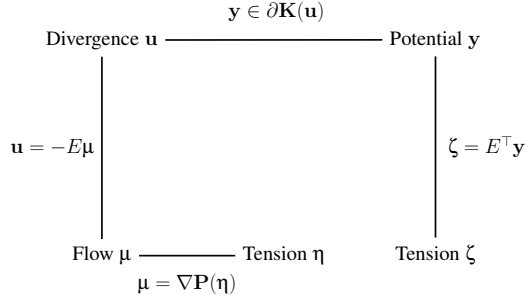}} 
\end{center}
\caption{The block diagram of the closed loop system (a) and the abstracted illustration of the network variables (b). }
\label{fig.Diagram}
\end{figure*}

\begin{table*}
\begin{center}
\begin{tabular}{|c l|c l|c|c|c|}
\hline
    \multicolumn{2}{|c|}{\textbf{Dynamic Signal}} & \multicolumn{2}{|c|}{\textbf{Network Variable}} & \textbf{Relation} & \textbf{Cost Function} & \textbf{Optimization Problem} \\
\hline 
\hline 
$\yd(t)$ & system output &  $\yb$ & potential &  $\yb = \kyf(\ub)$  & $K_{i}^{\star}(\yr_{i})$ &OPP1 \\ 
\hline 
$\zetad(t)$ & relative output & $\zetab$ & tension & $\zetab = E^{\top}\yb$ & $I_{0}(\zetar_{k})$ &OPP1 \\ 
\hline 
\hline 
$\ud(t)$ & system input & $\ub$ & divergence & $\ub = \kyf^{-1}(\yb)$ & $K_{i}(\ur_{i})$ &OFP1\\ 
\hline
$\mud(t)$ & controller output   & $\mub$ & flow & $\ub + E\mub = 0$ & $P_{k}^{\star}(\mur_{k})$ &OFP2 \\ 
\hline
\hline 
 $\vd(t)$ & --  & $\vb$ & potential & $\zb = E^{\top}\vb$ & $\ur_{k}\vr_{k}$ & OPP2\\ 
\hline
$\zd(t)$ & controller state  & $\zb$ & tension & $\mub = \bfpsi(\zb)$ & $P_{k}(\zr_{k})$ &OPP2\\ 
\hline 
\end{tabular}

\caption{Relation between variables involved in the dynamical system and their static counterparts. } \label{tab.Variables} 
\end{center}
\end{table*}

\subsection{The Closed-Loop Perspective}

Having established conditions that ensure the existence and the optimality properties of an output agreement steady-state solution, it remains to prove convergence. 
\begin{theorem}[{\small Output Agreement}]
Consider the dynamical network \eqref{sys.Node}, \eqref{eqn.Interconnectionzeta}, \eqref{sys.Controller1}, \eqref{eqn.Interconnectionu} and suppose that the nodes \eqref{sys.Node} are all output-strictly maximal passive with $\mc{U}_{i} = \mathbb{R}$ and $\mc{Y}_{i} = \mathbb{R}$ and all coupling nonlinearities $\psi_{k}$ are strongly monotone. Then there exist $\ub$, $\yb$, $\zb$, and $\mub$ being optimal solutions to \eqref{prob.OFP_Unconstrained}, \eqref{prob.PotentialUnconstrained2}, \eqref{prob.zExistence}, and \eqref{prob.ODP}, such that
$\lim_{t \rightarrow \infty} \ud(t) \rightarrow \ub$, $\lim_{t \rightarrow \infty} \yd(t) \rightarrow \yb$, $\lim_{t \rightarrow \infty} \zd(t) \rightarrow \zb$, and $\lim_{t \rightarrow \infty} \mud(t) \rightarrow \mub$. In particular, the dynamical network  converges to output agreement, i.e., $\lim_{t\rightarrow \infty} \yd(t) \rightarrow \beta \1$.
\end{theorem}
\begin{proof}
The assumptions ensure that the four network optimization problems \eqref{prob.OFP_Unconstrained}, \eqref{prob.PotentialUnconstrained2}, \eqref{prob.zExistence}, and \eqref{prob.ODP} have an optimal solution.
Thus, a steady-state solution exists. Output-strictly maximal equilibrium independent passivity of the node dynamics ensures that for all $i \in \Nodeset$ there exists a storage function $S_{i}$ such that $\dot{S}_{i} \leq - \rho_{i}\|y_{i}(t) - \yr_{i}\|^{2} + (y_{i}(t) - \yr_{i})(u_{i}(t) - \ur_{i})$. Additionally, maximal equilibrium independent passivity of the controller dynamics ensures that for all $k \in \Edgeset$ there exists a storage function $W_{k}$ satisfying $\dot{W}_{k} \leq (\mu_{k}(t) - \mur_{k})(\zeta_{k}(t) - \zetar_{k})$. 
Thus, the basic convergence result of Theorem \ref{thm.BasicConvergence} applies directly, proving convergence of the output trajectories, i.e., $\lim_{t \rightarrow \infty} \yd(t) \rightarrow \yb$.
Since $\yb \in \kyf(\ub)$, it follows that $\ud(t) $ must converge to $\ub$. The convergence of $\mud(t)$ and $\zetad(t)$ to $\mub$ and $\zetab$, respectively, follows immediately. 
\end{proof}

We can summarize the results of this section as follows.
All signals of the dynamical network \eqref{sys.Node}, \eqref{eqn.Interconnectionzeta}, \eqref{sys.Controller1}, \eqref{eqn.Interconnectionu} have static counterparts in the network optimization theory framework. 
The static counterparts of the outputs $\yd(t)$ are the solutions $\yb$ of an optimal potential problem \eqref{prob.PotentialUnconstrained2}. Equivalently, the corresponding dual variables, i.e., divergence variables in \eqref{prob.OFP_Unconstrained}, $\ub$, are the static counterparts to the control inputs $\ud(t)$.
The controller state $\zd(t)$ and the output $\mud(t)$ have the tension and flow variables of \eqref{prob.zExistence} and \eqref{prob.ODP}, respectively, as their static counterparts.
 We visualize the connection between the dynamic variables of the closed-loop system and the static network variables in Figure \ref{fig.Diagram}. Note that the signals in the dynamical system influence each other in a closed-loop structure, while there is no equivalent closed-loop relation for the network variables. In particular, the two tension variables $\zetab$ and $\zb$ are not connected, while their dynamic counterparts $\zetad(t)$ and $\zd(t)$ are connected by an integrator.
 Additionally, a summary of all variables involved in the output agreement problem together with their static counterparts is provided in Table \ref{tab.Variables}. For the sake of completeness, we include also the dynamic variable $\vd(t)$, which corresponds to the potential variables $\vb$ of \eqref{prob.zExistence}. Although we did not consider $\vd(t)$ explicitly in our discussion of the dynamical network, we can define it in accordance to \eqref{prob.zExistence} as $\zd(t) = E\vd(t)$.
 
%

%

\section{A General Dynamic Network Analysis Framework} \label{sec.NetworkAnalysis}

The full potential of the established duality framework can be seen if more general networks of maximal equilibrium independent passive systems are considered.
 A key component in the analysis of the previous section was that the controller dynamics \eqref{sys.Controller1} were maximal equilibrium independent passive systems. 
We will generalize the previous results for controllers  that are arbitrary  maximal equilibrium independent passive systems. In particular, we assume now that the controllers \eqref{sys.Controller1} are replaced by dynamical systems of the form

\begin{align}
\begin{split} \label{sys.Controller2}
\Pi_{k}: \quad \dot{\z}_{k} &= \phi_{k}(\z_{k},\zeta_{k}) \\
\mu_{k} &= \psi_{k}(\z_{k},\zeta_{k}), \quad k \in \Edgeset.
\end{split}
\end{align}
\begin{assumption}
The controllers \eqref{sys.Controller2} are maximal equilibrium independent passive with input set $\mc{Z}_{k}$, output set $\mc{M}_{k}$, and maximal monotone input-output relation $\gammar_{k} \subset \mathbb{R}^{2}$.
\end{assumption}

To each of the dynamics \eqref{sys.Controller2} one can associate now a closed, proper, convex function $\Gamma_{k}: \mathbb{R} \rightarrow \mathbb{R}$ such that
\begin{align}
\partial \Gamma_{k} = \gammar_{k}.
\end{align}

Now, the formalism developed in the previous section can be generalized as the asymptotic behavior of the network \eqref{sys.Node}, \eqref{eqn.Interconnectionzeta}, \eqref{sys.Controller2}, \eqref{eqn.Interconnectionu}  can be related to the following pair of dual network optimization problems.

 \textbf{Generalized Optimal Flow Problem} Consider the following \emph{optimal flow problem}
\begin{align}  \tag{GOFP}
\begin{split} \label{prob.GenOFP}
\min_{\ub,\mub} \; &\sum_{i=1}^{|\Nodeset|} K_{i}(\ur_{i}) + \sum_{k=1}^{|\Edgeset|}\Gamma_{k}^{\star}(\mur_{k}) \\
& \ub + E\mub = 0,
\end{split}
\end{align}
where $\Gamma_{k}^{\star}$ denotes the convex conjugate of $\Gamma_{k}$. This is a generalized version of \eqref{prob.OFP_Unconstrained}. Still the divergence $\ub$  are associated to the cost functions defined by the integral of the nodes input-output relations. However, now the cost function $\Gamma_{k}^{\star}$ is associated to the flow variables $\mur_{k}$.

 \textbf{Generalized Optimal Potential Problem} Dual to the generalized optimal flow problem, we also define the generalized optimal potential problem as
\begin{align}  \tag{GOPP}
\begin{split} \label{prob.GenOPP}
\min_{\yb,\zetab} \; & \sum_{i=1}^{|\Nodeset|} K_{i}^{\star}(\yr_{i})  + \sum_{k=1}^{|\Edgeset|} \Gamma_{k}(\zeta_{k}) \\
& \zetab = E^{\top}\yb.
\end{split}
\end{align}
In contrast to \eqref{prob.PotentialUnconstrained2}, this problem does not necessarily force the potential differences, i.e., the tensions, to be zero, but penalizes them with the general cost functions $\Gamma_{k}$.

The general network optimization problems \eqref{prob.GenOFP} and \eqref{prob.GenOPP} are related to the asymptotic behavior of the network of maximal equilibrium independent passive systems.

\begin{theorem}[{\small Generalized Network Convergence Theorem}]
 Consider the dynamical network \eqref{sys.Node}, \eqref{eqn.Interconnectionzeta},  \eqref{sys.Controller2}, \eqref{eqn.Interconnectionu}. Assume all node dynamics \eqref{sys.Node} are output strictly maximal equilibrium independent passive and all controller \eqref{sys.Controller2} dynamics are maximal equilibrium independent passive, and the two network optimization problems  \eqref{prob.GenOFP}, \eqref{prob.GenOPP} have a feasible solution. 
 Then there exists constant vectors $\ub$, $\mub$ solving \eqref{prob.GenOFP}, and $\yb$, $\zetab$ solving \eqref{prob.GenOPP}, such that
$\lim_{t \rightarrow \infty} \ud(t) \rightarrow \ub$, $\lim_{t \rightarrow \infty} \mud(t) \rightarrow \mub$, $\lim_{t \rightarrow \infty} \yd(t) \rightarrow \yb$, and $\lim_{t \rightarrow \infty} \zetad(t) \rightarrow \zetab$.
\end{theorem} 
\begin{proof}
First, we show that if the two network optimization problems have a feasible solution, this solution represents an equilibrium for the dynamical network.
Consider again the Lagrangian function of \eqref{prob.GenOFP} with Lagrange multiplier $\tilde{\yb}$, i.e.,
\[\mc{L}(\ub,\mub,\tilde{\yb}) = \sum_{i=1}^{|\Nodeset|} K_{i}(\ur_{i}) + \sum_{k=1}^{|\Edgeset|}\Gamma_{k}^{\star}(\mur_{k}) + \tilde{\yb}^{\top}( -\ub - E\mub).\] 
Define now $\tilde{\zetab} = E^{\top}\tilde{\yb}$. If \eqref{prob.GenOFP} has an optimal solution, this solution satisfies the optimality conditions
\begin{align}
\begin{split} \label{eqn.GenFOC}
\partial \mathbf{K}_{i}(\ub) - \tilde{\yb} \in 0, \quad
\partial \boldsymbol{\Gamma}^{\star}(\mub) - \tilde{\zetab} \in 0 \\
\ub + E\mub  = 0, \quad \tilde{\zetab} = E^{\top}\tilde{\yb},
\end{split}
\end{align}
where we use the notation $\boldsymbol{\Gamma}^{\star}(\mub) = \sum_{k=1}^{|\Edgeset|}\Gamma_{k}^{\star}(\mur_{k})$.
Since $\boldsymbol{\Gamma}(\zetab) = \sum_{k=1}^{|\Edgeset|}\Gamma_{k}(\zetar_{k})$ is a closed convex function it follows from the inversion of the subgradients (i.e., \cite[Thm. 23.5]{Rockafellar1997}) that $\partial \boldsymbol{\Gamma}^{\star}(\mub) $ is equivalent to $\mub \in \partial \boldsymbol{\Gamma}(\tilde{\zetab})$.
Thus, if \eqref{prob.GenOFP} has an optimal primal solution and dual solution, then these solutions are an equilibrium configuration of the dynamical network.
To complete this part of the proof, it remains to show that $\tilde{\yb}$ and $\tilde{\zetab}$ are optimal solutions to \eqref{prob.GenOPP}. Define $s(\tilde{\yb},\tilde{\zetab}) = \inf_{\ub,\mub} \mc{L}(\ub,\mub,\tilde{\yb})$ with $\tilde{\zetab} = E^{\top}\tilde{\yb}$. Clearly, $s(\tilde{\yb},\tilde{\zetab}) = - \sum_{i=1}^{|\Nodeset|} K_{i}^{\star}(\tilde{\yr}_{k}) - \sum_{k=1}^{|\Edgeset|}\Gamma_{k}^{\star\star}(\tilde{\zetar}_{k})$. Since $\Gamma_{k}^{\star\star} = \Gamma_{k}$ it can be readily seen that an optimal solution to $\inf_{\tilde{\yb}, \tilde{\zetab}}s(\tilde{\yb}\tilde{\zetab})$ is an optimal solution to \eqref{prob.GenOPP}. Thus, optimal solutions to \eqref{prob.GenOFP}, \eqref{prob.GenOPP} are equilibrium configurations for the network. By the same argument follows that all possible network equilibrium configurations are solution to \eqref{prob.GenOFP}, \eqref{prob.GenOPP}.

It remains to prove convergence. Consider an network equilibrium configuration  $\ub$, $\yb$, $\mub$, and $\zetab$. By assumption, the node dynamics are output strictly maximal equilibrium independent passive and since  \eqref{prob.GenOFP}, \eqref{prob.GenOPP} are feasibel $\ub \in \mc{U}$ and $\yb \in \mc{Y}$. All controllers are maximal equilibrium independent passive and since  \eqref{prob.GenOFP}, \eqref{prob.GenOPP} are feasible, $\mub \in \mc{M}$, and $\zetab \in \mc{Z}$. Convergence of the trajectories follows now from the basic convergence result.
\end{proof}

\begin{remark}[{\small Revisiting Output Agreement}]
The general result includes the output agreement problem studied in the previous section. There, the equilibrium input-output relation $\gammar_{k}$ of the controller \eqref{sys.Controller1} is the vertical line through the origin, such that $\Gamma_{k}$ is the indicator function for the origin.
Its convex conjugate is $\Gamma_{k}^{\star}(\mur) = 0$. Now, \eqref{prob.GenOFP} and \eqref{prob.GenOPP} reduce to the original problems \eqref{prob.OFP_Unconstrained} and \eqref{prob.PotentialUnconstrained2}. \\
\end{remark}

\section{Application: Analysis of a Traffic Dynamics Model} \label{sec.Applications}
The potential of the proposed network optimization interpretation is now illustrated on the analysis of a nonlinear traffic dynamics models.
The considered model is an \emph{optimal velocity model}, as proposed in \cite{Bando1995}, \cite{Helbing1998}, with the following assumtions: (i) the drivers are heterogeneous and have different ``preferred" velocities, (ii) the influence between cars is bi-directional, and (iii) vehicles can overtake other vehicles.
Each vehicle adjusts its velocity $v_{i}$ according to
\begin{align} \label{sys.Traffic1}
\dot{v}_{i} = \kappa_{i}[V_{i}(\Delta p) - v_{i}],
\end{align} 
where $\kappa_{i} > 0$ is a constant and the adjustment $V_i(\Delta p)$ depends on the relative position to other vehicles, i.e., $\Delta p = p_{j} - p_{i}$, as
\begin{align} \label{sys.Traffic2}
 V_{i}(\Delta p) = V_{i}^{0} + V_{i}^{1}\sum_{j \in \mc{N}(i)} \tanh(p_{j} - p_{i}).
\end{align}
Here $\mathcal{N}(i)$ is used to denote the neighboring vehicles influencing vehicle $i$. Throughout this example we assume that the set of neighbors to a vehicle is not changing over time.
 The constants $V_{i}^{0} > 0$ are ``preferred velocities" and $V_{i}^{1} > 0$ are ``sensitivities'' of the drivers. In the following we assume $V_{i}^{0} \neq V_{j}^{0}$ (i.e., heterogeneity). 

The model can be represented in the form \eqref{sys.Node}, \eqref{eqn.Interconnectionzeta}, \eqref{eqn.Interconnectionu}, \eqref{sys.Controller2}.
The node dynamics can be identified as 
\begin{align}
\dot{v}_{i}(t) = \kappa_{i}[- v_{i}(t) + V_{i}^{0} + V_{i}^{1} u_{i}(t)],\; y_{i}(t) = v_{i}(t),
\end{align}
 with the velocity $v_{i}(t)$ being the node state. The input to each vehicle computes as $u_{i}(t) := \sum_{j \in \mc{N}(i)} \tanh(p_{j}(t) - p_{i}(t))$.  The relative velocities of neighboring vehicles are $\zetad(t) = E^{\top}\yd$. Now, since $\dot{p}_{i} = v_{i}$, we can define the relative positions of neighboring  vehicles as $\z_{k}(t)   = p_{j}(t) - p_{i}(t)$, where edge $k $ connects nodes $i$ and $j$. In vector notation, the coupling can be represented as
 \begin{align}
 \begin{split} \label{sys.TrafficCoupling}
 \dot{\zd} &= \zetad\\
 \mud &= \tanh(\zd),
 \end{split}
 \end{align}
 and $\ud = -E\mud$, where $\tanh(\cdot)$ is here the vector valued function with each entry being the $\tanh$ of the respective entry of $\zd$.

 The node dynamics are output strictly maximal equilibrium independent passive systems. The equilibrium input-output map is the affine function
 $k_{\yr,i}(\ur_{i}) = V_{i}^{0} +V_{i}^{1}\ur_{i}$
 and a corresponding storage function is $S_{i} = \frac{1}{2\kappa_{i} V_{i}^{1}}(v_{i}(t) - \mathrm{v}_{i})^{2}$, where $\mathrm{v}_{i}$ is the desired constant velocity.
The objective functions associated to the node dynamics are the quadratic functions
\begin{align}
K_{i}(\ur_{i}) = \frac{V_{i}^{1}}{2}\ur_{i}^{2} + V_{i}^{0}\ur_{i} \quad \mbox{and} \quad K^{\star}_{i}(\yr_{i}) = \frac{1}{2V_{i}^{1}}(\yr_{i} - V_{i}^{0})^{2}.
\end{align}

\begin{figure*}[t!]
\begin{center}
\subfigure[$P^{\star}(\mur)$]{\scalebox{0.4}{\input{P_star.tikz}}}
\subfigure[$\mur = \nabla P(\z) \; (=: \psi(\z))$]{\scalebox{0.4}{\input{tanh.tikz}}}
\subfigure[$P(\z)$]{\scalebox{0.4}{\input{P.tikz}}}
\end{center}
\caption{ Relation between the \emph{flow cost function} $P^{\star}(\mu)$, the \emph{coupling nonlinearity}, here $\psi(\z) = \nabla P(\z) := \tanh(\z)$, and the \emph{coupling function integral} $P(\z)$.}
\label{fig.Correspondence}
\end{figure*}
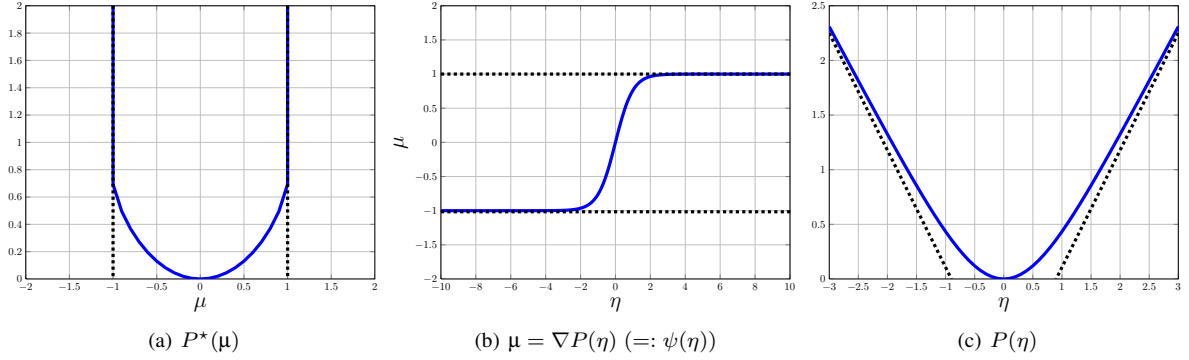

Next, we show that the controller dynamics \eqref{sys.TrafficCoupling} is maximal equilibrium independent passive. Note that the output functions of \eqref{sys.TrafficCoupling} are monotone but bounded. The dynamics \eqref{sys.TrafficCoupling} will only attain a steady state for $\zetad = 0$. However, if $\zetad \neq 0$, the outputs will not grow unbounded, but will approach the saturation bounds of the nonlinearity. Thus, each of the coupling dynamics has the equilibrium input-output relation
\begin{align} \label{eqn.ClusteringIO}
\gammar_{k}(\zetar_{k})  =  \begin{cases}
+1 & \zetar_{k} > 0 \\[-0.5em]
(-1,1) & \zetar_{k} = 0 \\[-0.5em]
-1 & \zetar_{k} < 0.
\end{cases}
\end{align} 
It can be easily verified that $\gammar_{k}$ represents a maximal monotone relation in $\mathbb{R}^{2}$.
To prove now maximal equilibrium independent passivity, we define the integral functions of the coupling nonlinearities, i.e., $P_{k}(\zr_{k}) = \ln \cosh(\zr_{k})$. Note that the functions $P_{k}$ are not strongly convex, as they asymptotically approach an affine function. 
The function $P_{k}$, the coupling nonlinearity  $\psi_{k}(\z_{k}) = \mathrm{tanh}(\z_{k})$, and the convex conjugate $P_{k}^{\star}(\mu_{k})$ are illustrated in Figure \ref{fig.Correspondence}. 

The function $P_{k}$ can now be used to prove maximal equilibrium independent passivity.

\begin{proposition}
Each of the dynamics \eqref{sys.TrafficCoupling} is maximal equilibrium independent passive with equilibrium input output relation \eqref{eqn.ClusteringIO}.
\end{proposition}
\begin{proof}
For any $\zetar_{k} = 0$ and any $\mur_{k} \in (-1,1)$, there is a unique $\zr_{k}$ such that $\mur_{k} = \tanh(\zr_{k})$. The corresponding storage function 
\begin{align} \label{eqn.Bregman}
 W_{k}(\z_{k}(t)) = P_{k}(\z_{k}(t)) - P_{k}(\zr_{k}) - \nabla P_{k}(\zr_{k})(\z(t) - \zr_{k})
 \end{align}
is positive definite. It can be readily seen that in this case $\dot{W} =  (\mu_{k}(t) - \mur_{k})\zeta_{k}(t) =  (\mu_{k}(t) - \mur_{k})(\zeta_{k}(t) - \zetar_{k})$.
Furthermore, if $\zeta_{r} \neq 0$ we can define a sequence $\zr_{k}^{1}, \zr_{k}^{2}, \ldots $ that diverges to $+\infty$ if $\zeta_{r} > 0$ and to $-\infty$ if $\zeta_{r} < 0$. To each $\zr_{k}^{\ell}$ one can define the positive definite function \eqref{eqn.Bregman}, named $W_{k}^{\ell}(\z_{k})$. The sequence of functions $W_{k}^{\ell}$ approaches a positive semidefinite function $\bar{W}_{k}(\z_{k})$ that satisfies
\[ \dot{\bar{W}}_{k} = (\mu_{k}(t) - \mur_{k})\zeta_{k}(t).  \]
Additionally, we note that if $\zetar_{k} > 0$ ($\zetar_{k} < 0$) then $(\mu_{k}(t) - \mur_{k}) \leq0$ ($(\mu_{k}(t) - \mur_{k}) \geq0$) for all $\mu_{k}$.  Thus, it holds that $\zetar_{k}(\mu_{k}(t) - \mur_{k}) \leq 0$. Based on this observation we conclude
\[ \dot{\bar{W}}_{k} \leq (\mu_{k}(t) - \mur_{k})(\zeta_{k}(t) - \zetar_{k}).  \]
Thus, for each $\zetar_{k}$ and $\mur_{k} \in \gamma_{k}(\zetar_{k})$, there exists a positive semidefinite storage function that allows to conclude passivity. 
\end{proof}
Thus, the nonlinear traffic dynamic model \eqref{sys.Traffic1}, \eqref{sys.Traffic2} can be understood as the feedback interconnection of an output strictly maximal equilibrium independent passive system with a maximal equilibrium independent passive controller.

In the network optimization interpretation are the \emph{potential} variables the velocities, i.e., $v_{i}$, and the \emph{divergence} variables are the influence of the other vehicles. 
Furthermore, the \emph{tensions} are the relative positions of the vehicles and the \emph{flows} are their mapping through the coupling functions.
To complete the network theoretic interpretation of the traffic dynamics model, we define the integral function of the input output relation $\gamma_{k}$. 
The integral function of $\gamma_{k}(\zetar_{k})$ is the absolute value of $\zetar_{k}$ and its convex conjugate is the indicator function for the set $ [-1, 1]$, i.e., 
\[ \Gamma_{k}(\zetar_{k}) = |\zetar_{k}|, \quad \Gamma_{k}^{\star}(\mur_{k}) = I_{[-1,1]}(\mur_{k}).\]
Thus, for the traffic dynamics, the two network optimization problems \eqref{prob.GenOFP} and \eqref{prob.GenOPP} take a very characteristic structure.
The optimal flow problem \eqref{prob.GenOFP} is almost identical to \eqref{prob.OFP_Unconstrained}, except that additionally constraints on the flow variables are imposed, i.e., the flows are constrained as $-1 \leq \mur_{k} \leq 1$.
On the other hand, the optimal potential problem \eqref{prob.GenOPP} has a quadratic cost function for the potentials plus an additional absolute value of the potential differences, that can be understood as an $\ell_{1}$-penalty.

\begin{figure*}[t!]
\begin{minipage}{\textwidth}
\begin{minipage}{0.48\textwidth}
\includegraphics[ trim= 2cm 8.5cm 2.5cm 9cm, width = \textwidth]{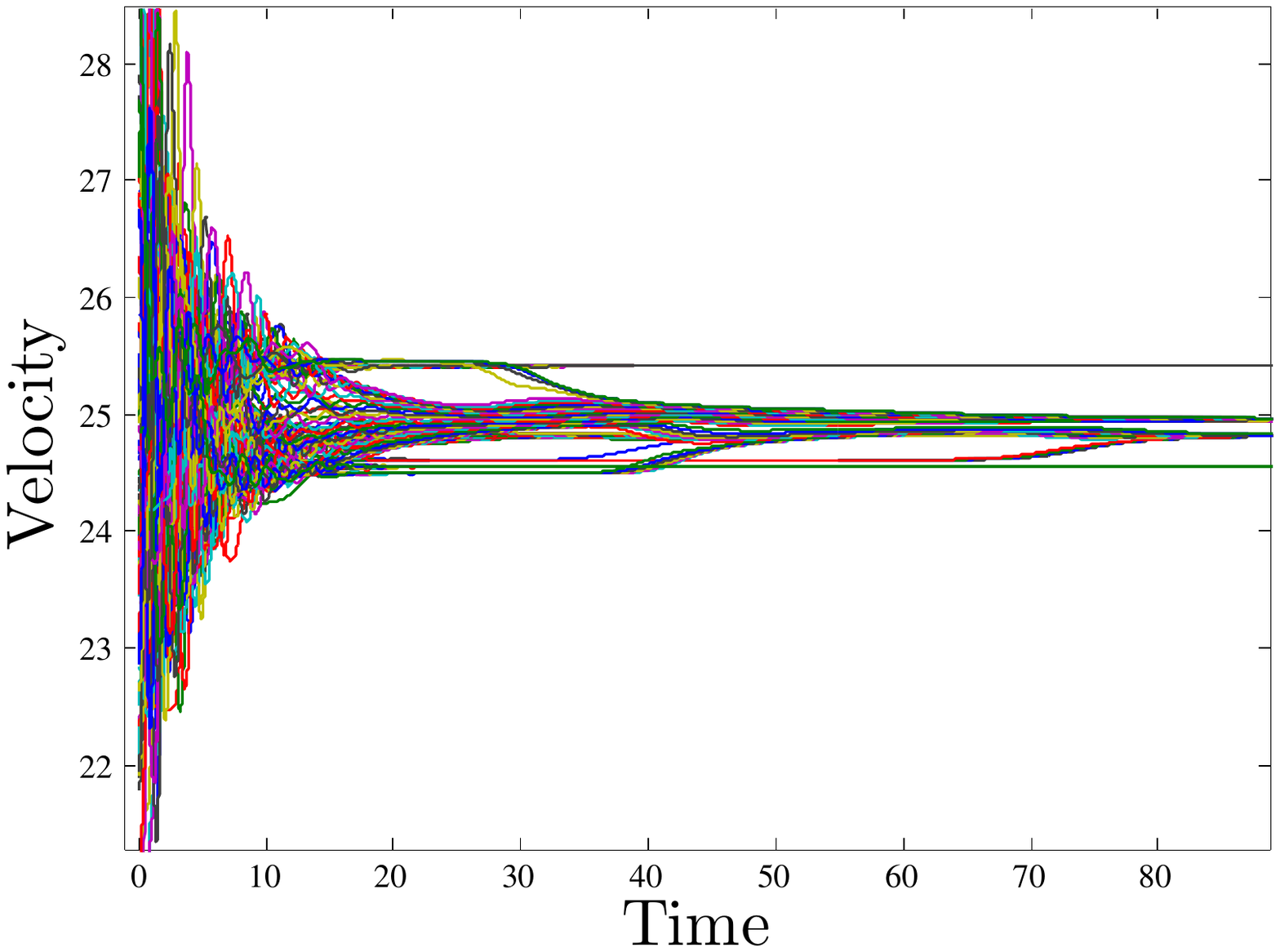}
\end{minipage}
\hfill
\begin{minipage}{0.49\textwidth}
\scalebox{0.6}{
\input{TrafficStructure.tikz} }
\end{minipage}
\end{minipage}
\caption{Simulation results for a traffic flow model with 100 vehicles placed on a line graph. \emph{Left:} Time trajectories of the velocities for normally distributed coefficients with $\sigma^{0} = 2.5$ and $\sigma^{1} =1$.
\emph{Right:} Asymptotic velocities predicted by the network optimization problems for $\sigma^{0} = 1$ (blue, '$\square$'), $\sigma^{0} = 2.5$ (red, 'o'), and $\sigma^{0} = 4$ (green, '$\Delta$'). }
\label{fig.TrafficFlow}
\end{figure*}

\begin{remark}[{\small Network Clustering}] The connection of the presented results to the network clustering analysis presented in \cite{Burger2011a} can be explained on the traffic dynamics. 
In \cite{Burger2011a} a \emph{saddle-point problem} is proposed to analyze and predict an asymptotic clustering behavior. 
%
%
The saddle-point problem of \cite{Burger2011a} for the traffic dynamics studied here (using the notation of this paper) is
\begin{align*} 
\begin{split}
\max_{\mur_{k}} \; \min_{\yr_{i}} \; \mc{L}(\yb, \mub) := &\sum_{i=1}^{|\Nodeset|}  \frac{1}{2V_{i}^{1}}(\yr_{i} - V_{i}^{0})^{2}   \; + \; \mub^{\top} E^{\top}\yb \\
&-1 \leq \mur_{k} \leq 1.
\end{split}
\end{align*}
Some straight forward manipulations reveal that the saddle-point problem results in fact from the Lagrange dual of \eqref{prob.GenOPP} for the traffic dynamics model. The saddle point problem involves variables from the two dual network optimization problems, i.e., the potential variables $\yr_{i}$ and the flow variables $\mur_{k}$.
It has been shown in \cite{Burger2011a} that the solutions to the saddle-point problem eventually have a clustered structure. Thus, one might expect a clustering behavior to happen in the traffic dynamics model and this behavior can in fact be observed in the simulations. We refer to \cite{Burger2011a} for a detailed treatment of the clustering phenomenon.
\end{remark}

We present a computational study  with 100 vehicles placed on a line graph in Figure \ref{fig.TrafficFlow}. The sensitivity parameter is $\kappa = 0.6$ for all vehicles, while the parameters $V_{i}^{0}$ and $V_{i}^{1}$ are chosen as a common nominal parameter plus a random component, i.e., $V_{i}^{0} =  V_{nom}^{0} + V_{i,rand}^{0}$.  The common off-set is $V_{nom}^{0} = 25 \frac{m}{s}$ and $V_{1} = 10 \frac{m}{s}$. The random component is chosen according to a zero mean normal distribution with different standard deviations.
In Figure \ref{fig.TrafficFlow} (left), the time-trajectories of the velocities $v_{i}$ are shown with the random coefficients $V_{i,rand}^{0}, V_{i,rand}^{1}$ chosen from a distribution with $\sigma^{0} = 2.5$ and $\sigma^{1} = 1$, respectively.
Figure \ref{fig.TrafficFlow} (right) shows the asymptotic velocity distribution of the traffic for different choices of the standard deviation $\sigma^{0}$. While for $\sigma^{0} =1$ the traffic agrees on a common velocity, already for $\sigma^{0} = 2.5$ a clustering structure of the network can be seen. The clustering structure becomes more refined for $\sigma^{0}=4$. We have chosen for all studies $\sigma^{1} =1$.
Please note that the novel network theoretic framework provides us with efficient tools to analyze and predict the non-trivial asymptotic behavior of the nonlinear traffic dynamics, without the need to simulate the system for different parameter configurations.

\section{Conclusions} \label{sec.Conclusion}

We have established in this paper an intimate connection between passivity-based cooperative control and the network optimization theory of Rockafellar \cite{Rockafellar1998}. 
To obtain this connection, we introduced the notion of maximal equilibrium independent passivity as a variation of the equilibrium independent passivity concept of \cite{Hines2011}.
It was shown that dynamical networks involving maximal equilibrium independent passive systems asymptotically approach the solutions of several network optimization problems.
For output agreement problems we have shown that it is a necessary condition that the output agreement steady-state is optimal with respect to an optimal flow and an optimal potential problem. This connection provided also an interpretation of the system outputs as potential variables, and of the system inputs as node divergence.
Similar inverse optimality and duality results are established for general networks of maximal equilibrium independent passive systems, that do not necessarily converge to output agreement. The general theory was illustrated on a nonlinear traffic dynamics model that shows asymptotically a clustering behavior.
As maximal equilibrium independent passive systems admit a certain inverse optimality and a strong duality, the presented results suggest that networks consisting of those systems are ``well-behaved'' and ``easy'' to be analyzed.
We believe that this result contributes to a unified understanding of networked dynamical systems and opens the way for further advanced analysis methods.

%
%
%
%
%
%

{\footnotesize
 \bibliographystyle{IEEEtran}
\bibliography{DualConstraints}
}

\end{document}

%% file: P_star.tikz
%
%
%
%
\begin{tikzpicture}

\begin{axis}[%
width=4.52083333333333in,
height=3.565625in,
scale only axis,
xmin=-2, xmax=2,
xmajorgrids,
ymin=0, ymax=2,
ymajorgrids,
xlabel={{\huge $\mu$}},
]
\addplot [
color=blue,
solid,
line width=3.0pt
]
coordinates{
 (-0.999999,0.693139426236731)(-0.9,0.494631937214073)(-0.8,0.368064207168497)(-0.7,0.270438092753954)(-0.6,0.192744757021757)(-0.5,0.130812035941137)(-0.4,0.0822828785050518)(-0.3,0.0457005415253128)(-0.2,0.0201355135506888)(-0.1,0.00500836684635683)(0,0)(0.1,0.00500836684635683)(0.2,0.0201355135506888)(0.3,0.0457005415253128)(0.4,0.0822828785050518)(0.5,0.130812035941137)(0.6,0.192744757021757)(0.7,0.270438092753954)(0.8,0.368064207168497)(0.9,0.494631937214073)(0.9999999,0.693146290017787) 
};

\end{axis}

\begin{axis}[%
width=4.52083333333333in,
height=3.565625in,
scale only axis,
xmin=0, xmax=1,
ymin=0, ymax=1,
hide x axis,
hide y axis]
\addplot [
color=blue,
solid,
line width=3.0pt
]
coordinates{
 (0.75,1)(0.75,0.34) 
};

\addplot [
color=blue,
solid,
line width=3.0pt
]
coordinates{
 (0.25,1)(0.25,0.34)
};

\addplot [
color=black,
dashed,
line width=3.0pt]
coordinates{
 (0.25,1)(0.25,0)
};

\addplot [
color=black,
dashed,
line width=3.0pt]
coordinates{
 (0.75,1)(0.75,0)
};
\end{axis}
\end{tikzpicture}

%% file: tanh.tikz
%
%
%
%
\begin{tikzpicture}

\begin{axis}[%
width=4.52083333333333in,
height=3.565625in,
scale only axis,
xmin=-10, xmax=10,
xmajorgrids,
ymin=-2, ymax=2,
ymajorgrids,
xlabel={{\huge $\z$}},
ylabel={{\huge $\mu$}}]
]
\addplot [
color=blue,
solid,
line width=3.0pt
]
coordinates{
 (-10,-0.999999995877693)(-9.9,-0.999999994965003)(-9.8,-0.99999999385024)(-9.7,-0.999999992488666)(-9.6,-0.999999990825637)(-9.5,-0.999999988794407)(-9.4,-0.999999986313458)(-9.3,-0.99999998328322)(-9.2,-0.999999979582079)(-9.1,-0.999999975061495)(-9,-0.999999969540041)(-8.9,-0.999999962796122)(-8.8,-0.999999954559081)(-8.7,-0.999999944498337)(-8.6,-0.999999932210116)(-8.5,-0.999999917201249)(-8.4,-0.999999898869378)(-8.3,-0.999999876478781)(-8.2,-0.999999849130844)(-8.1,-0.999999815728)(-8,-0.999999774929676)(-7.9,-0.999999725098492)(-7.8,-0.99999966423455)(-7.7,-0.999999589895169)(-7.6,-0.999999499096851)(-7.5,-0.999999388195546)(-7.4,-0.999999252740403)(-7.3,-0.999999087295143)(-7.2,-0.999998885219883)(-7.1,-0.999998638404658)(-7,-0.999998336943945)(-6.9,-0.999997968739121)(-6.8,-0.999997519012918)(-6.7,-0.999996969716367)(-6.6,-0.999996298804454)(-6.5,-0.999995479351404)(-6.4,-0.9999944784701)(-6.3,-0.999993255992273)(-6.2,-0.99999176285651)(-6.1,-0.999989939139396)(-6,-0.999987711650796)(-5.9,-0.999984990996806)(-5.8,-0.99998166799256)(-5.7,-0.99997760928099)(-5.6,-0.999972651981831)(-5.5,-0.999966597156304)(-5.4,-0.99995920182544)(-5.3,-0.999950169222121)(-5.2,-0.999939136886199)(-5.1,-0.999925662125794)(-5,-0.999909204262595)(-4.9,-0.999889102950554)(-4.8,-0.999864551700761)(-4.7,-0.999834565554297)(-4.6,-0.999797941612184)(-4.5,-0.999753210848027)(-4.4,-0.999698579283881)(-4.3,-0.999631856190073)(-4.2,-0.999550366459533)(-4.1,-0.999450843687797)(-4,-0.999329299739067)(-3.9,-0.999180865670028)(-3.8,-0.998999597785841)(-3.7,-0.998778241281131)(-3.6,-0.998507942332327)(-3.5,-0.998177897611199)(-3.4,-0.997774927934279)(-3.3,-0.997282960099142)(-3.2,-0.996682397839651)(-3.1,-0.9959493592219)(-3,-0.99505475368673)(-2.9,-0.993963167350583)(-2.8,-0.992631520201128)(-2.7,-0.991007453678118)(-2.6,-0.989027402201099)(-2.5,-0.98661429815143)(-2.4,-0.98367485769368)(-2.3,-0.980096396266191)(-2.2,-0.975743130031452)(-2.1,-0.970451936613454)(-2,-0.964027580075817)(-1.9,-0.956237458127739)(-1.8,-0.946806012846268)(-1.7,-0.935409070603099)(-1.6,-0.921668554406471)(-1.5,-0.905148253644866)(-1.4,-0.885351648202263)(-1.3,-0.861723159313306)(-1.2,-0.833654607012155)(-1.1,-0.80049902176063)(-1,-0.761594155955765)(-0.9,-0.716297870199025)(-0.799999999999999,-0.664036770267848)(-0.699999999999999,-0.604367777117163)(-0.6,-0.537049566998035)(-0.5,-0.46211715726001)(-0.399999999999999,-0.379948962255224)(-0.299999999999999,-0.29131261245159)(-0.199999999999999,-0.197375320224903)(-0.0999999999999996,-0.0996679946249555)(0,0)(0.0999999999999996,0.0996679946249555)(0.199999999999999,0.197375320224903)(0.299999999999999,0.29131261245159)(0.399999999999999,0.379948962255224)(0.5,0.46211715726001)(0.6,0.537049566998035)(0.699999999999999,0.604367777117163)(0.799999999999999,0.664036770267848)(0.9,0.716297870199025)(1,0.761594155955765)(1.1,0.80049902176063)(1.2,0.833654607012155)(1.3,0.861723159313306)(1.4,0.885351648202263)(1.5,0.905148253644866)(1.6,0.921668554406471)(1.7,0.935409070603099)(1.8,0.946806012846268)(1.9,0.956237458127739)(2,0.964027580075817)(2.1,0.970451936613454)(2.2,0.975743130031452)(2.3,0.980096396266191)(2.4,0.98367485769368)(2.5,0.98661429815143)(2.6,0.989027402201099)(2.7,0.991007453678118)(2.8,0.992631520201128)(2.9,0.993963167350583)(3,0.99505475368673)(3.1,0.9959493592219)(3.2,0.996682397839651)(3.3,0.997282960099142)(3.4,0.997774927934279)(3.5,0.998177897611199)(3.6,0.998507942332327)(3.7,0.998778241281131)(3.8,0.998999597785841)(3.9,0.999180865670028)(4,0.999329299739067)(4.1,0.999450843687797)(4.2,0.999550366459533)(4.3,0.999631856190073)(4.4,0.999698579283881)(4.5,0.999753210848027)(4.6,0.999797941612184)(4.7,0.999834565554297)(4.8,0.999864551700761)(4.9,0.999889102950554)(5,0.999909204262595)(5.1,0.999925662125794)(5.2,0.999939136886199)(5.3,0.999950169222121)(5.4,0.99995920182544)(5.5,0.999966597156304)(5.6,0.999972651981831)(5.7,0.99997760928099)(5.8,0.99998166799256)(5.9,0.999984990996806)(6,0.999987711650796)(6.1,0.999989939139396)(6.2,0.99999176285651)(6.3,0.999993255992273)(6.4,0.9999944784701)(6.5,0.999995479351404)(6.6,0.999996298804454)(6.7,0.999996969716367)(6.8,0.999997519012918)(6.9,0.999997968739121)(7,0.999998336943945)(7.1,0.999998638404658)(7.2,0.999998885219883)(7.3,0.999999087295143)(7.4,0.999999252740403)(7.5,0.999999388195546)(7.6,0.999999499096851)(7.7,0.999999589895169)(7.8,0.99999966423455)(7.9,0.999999725098492)(8,0.999999774929676)(8.1,0.999999815728)(8.2,0.999999849130844)(8.3,0.999999876478781)(8.4,0.999999898869378)(8.5,0.999999917201249)(8.6,0.999999932210116)(8.7,0.999999944498337)(8.8,0.999999954559081)(8.9,0.999999962796122)(9,0.999999969540041)(9.1,0.999999975061495)(9.2,0.999999979582079)(9.3,0.99999998328322)(9.4,0.999999986313458)(9.5,0.999999988794407)(9.6,0.999999990825637)(9.7,0.999999992488666)(9.8,0.99999999385024)(9.9,0.999999994965003)(10,0.999999995877693) 
};

\end{axis}

\begin{axis}[%
width=4.52083333333333in,
height=3.565625in,
scale only axis,
xmin=0, xmax=1,
ymin=0, ymax=1,
hide x axis,
hide y axis]
\addplot [
color=black,
dashed,
line width=3.0pt
]
coordinates{
 (0,0.75)(4.5,0.75) 
};

\addplot [
color=black,
dashed,
line width=3.0pt
]
coordinates{
 (0,0.246)(4.5,0.246) 
};
\end{axis}
\end{tikzpicture}

%% file: P.tikz
%
%
%
%
\begin{tikzpicture}

\begin{axis}[%
width=4.52083333333333in,
height=3.565625in,
scale only axis,
xmin=-3, xmax=3,
xmajorgrids,
ymin=0, ymax=2.5,
ymajorgrids,
xlabel={{\huge $\z$}},
]
\addplot [
color=blue,
solid,
line width=3.0pt
]
coordinates{
 (-3,2.30932850457779)(-2.9,2.20987580037089)(-2.8,2.110543862867)(-2.7,2.0113592312393)(-2.6,1.91235422334971)(-2.5,1.81356816792917)(-2.4,1.71504888677832)(-2.3,1.61685447149571)(-2.2,1.51905540404775)(-2.1,1.42173707411197)(-2,1.32500274735786)(-1.9,1.22897703589493)(-1.8,1.13380991244826)(-1.7,1.03968128986492)(-1.6,0.946806152602485)(-1.5,0.855440171013797)(-1.4,0.765885645728026)(-1.3,0.678497511407724)(-1.2,0.593688971594004)(-1.1,0.51193613920875)(-1,0.433780830483027)(-0.9,0.359830429966129)(-0.8,0.290753560328393)(-0.7,0.227270229358505)(-0.6,0.170135286778086)(-0.5,0.120114506958277)(-0.4,0.0779534853878323)(-0.3,0.0443407699259403)(-0.2,0.0198680718400074)(-0.0999999999999996,0.00499168882164644)(0,0)(0.0999999999999996,0.00499168882164644)(0.2,0.0198680718400074)(0.3,0.0443407699259403)(0.4,0.0779534853878323)(0.5,0.120114506958277)(0.6,0.170135286778086)(0.7,0.227270229358505)(0.8,0.290753560328393)(0.9,0.359830429966129)(1,0.433780830483027)(1.1,0.51193613920875)(1.2,0.593688971594004)(1.3,0.678497511407724)(1.4,0.765885645728026)(1.5,0.855440171013797)(1.6,0.946806152602485)(1.7,1.03968128986492)(1.8,1.13380991244826)(1.9,1.22897703589493)(2,1.32500274735786)(2.1,1.42173707411197)(2.2,1.51905540404775)(2.3,1.61685447149571)(2.4,1.71504888677832)(2.5,1.81356816792917)(2.6,1.91235422334971)(2.7,2.0113592312393)(2.8,2.110543862867)(2.9,2.20987580037089)(3,2.30932850457779) 
};

\end{axis}

\begin{axis}[%
width=4.52083333333333in,
height=3.565625in,
scale only axis,
xmin=0, xmax=1,
ymin=0, ymax=1,
hide x axis,
hide y axis]
\addplot [
color=black,
line width=3.0pt,
dashed
]
coordinates{
 (1,0.9)(0.65,0) 
};

\addplot [
color=black,
line width=3.0pt,
dashed
]
coordinates{
 (0,0.9)(0.35,0) 
};

\end{axis}
\end{tikzpicture}

%% file: TrafficStructure.tikz
%
%
%
%
\begin{tikzpicture}

\begin{axis}[%
width=4.52083333333333in,
height=3.565625in,
scale only axis,
xmin=0, xmax=100,
xlabel={ {\Large Vehicle Number}},
ymin=23, ymax=26.5,
ylabel={ {\Large Velocity}},
axis on top]

\addplot [
color=blue,
mark size=4pt,
only marks,
mark=square*,
mark options={solid, fill=blue}
]
coordinates{
 (1,24.9619991790571)(2,24.9619990838253)(3,24.9619988551422)(4,24.9619982648045)(5,24.9619979081799)(6,24.9619967465121)(7,24.9619957509172)(8,24.9619944679889)(9,24.9619934093)(10,24.9619909934472)(11,24.9619889524382)(12,24.9619860850079)(13,24.9619835196488)(14,24.961979888268)(15,24.9619744940127)(16,24.9619693348793)(17,24.9619646949074)(18,24.9619608497499)(19,24.9619571932383)(20,24.9619539620336)(21,24.9619521960556)(22,24.9619505571222)(23,24.9619487116969)(24,24.9619472997342)(25,24.9619462351629)(26,24.9619454732507)(27,24.9619447120907)(28,24.9619443338737)(29,24.9619445581865)(30,24.9619440600726)(31,24.9619438076164)(32,24.961943318827)(33,24.9619433004907)(34,24.961942092135)(35,24.961940853388)(36,24.9619388002871)(37,24.9619358968963)(38,24.9619329319966)(39,24.9619295846056)(40,24.9619268043354)(41,24.9619238538476)(42,24.9619196455766)(43,24.9619161135948)(44,24.9619131608735)(45,24.9619103601838)(46,24.9619067145864)(47,24.9619041431187)(48,24.9619023511529)(49,24.9619007215877)(50,24.9618992281467)(51,24.9618973325464)(52,24.9618947048308)(53,24.9618927513775)(54,24.9618919678106)(55,24.9618919216885)(56,24.9618922972901)(57,24.9618932399471)(58,24.9618942028335)(59,24.9618953400309)(60,24.9618955584807)(61,24.9618954279381)(62,24.9618952891463)(63,24.9618956183211)(64,24.961895870658)(65,24.9618959660218)(66,24.9618958778993)(67,24.9618959111101)(68,24.9618963405921)(69,24.9618960303437)(70,24.9618959292979)(71,24.9618957965028)(72,24.9618953043579)(73,24.9618947836886)(74,24.9618945482071)(75,24.9618934228847)(76,24.961893517034)(77,24.9618933695182)(78,24.9618925661137)(79,24.9618918148975)(80,24.9618911001991)(81,24.9618906489046)(82,24.9618887894165)(83,24.9618857485095)(84,24.96188271537)(85,24.9618799065161)(86,24.9618777758025)(87,24.9618736405187)(88,24.9618712239262)(89,24.9618693089633)(90,24.9618676435174)(91,24.9618666215172)(92,24.9618661378702)(93,24.9618653440941)(94,24.9618645790357)(95,24.9618640151826)(96,24.9618635400819)(97,24.9618630956889)(98,24.9618630566908)(99,24.9618630980926) 
};
\addlegendentry{$\sigma^{0} = 1$}

\addplot [
color=red,
only marks,
mark size=4pt,
mark=*,
mark options={solid, fill=red}
]
coordinates{
 (1,25.4099696461277)(2,25.4099696410878)(3,25.4099696213093)(4,25.4099695152627)(5,25.4099695080385)(6,25.4099692682362)(7,25.4099691244156)(8,25.4099689178977)(9,25.409968813063)(10,25.4099681591659)(11,25.4099678128099)(12,25.4099670927749)(13,25.4099666560779)(14,25.409965638807)(15,24.8159271147601)(16,24.8159139067431)(17,24.8159108273287)(18,24.8159096782795)(19,24.815908683203)(20,24.8159079940467)(21,24.815907847032)(22,24.8159077250469)(23,24.8159075415976)(24,24.8159074591377)(25,24.8159074552179)(26,24.8159075254368)(27,24.8159075898601)(28,24.8159077762913)(29,24.8159083717888)(30,24.8159085045169)(31,24.8159087283052)(32,24.8159088603008)(33,24.8159092087054)(34,24.8159091402974)(35,24.8159090628698)(36,24.8159087482374)(37,24.8159079629043)(38,24.8159071227529)(39,24.8159057797891)(40,24.8159050975676)(41,24.8159042688201)(42,24.8158360288323)(43,24.8158341367371)(44,24.8158332718335)(45,24.8158325297926)(46,24.8158298473518)(47,24.815829248797)(48,24.8158289994227)(49,24.8158287951427)(50,24.8158286233346)(51,24.8158283087459)(52,24.8158275408195)(53,24.8158271798678)(54,24.8158271758257)(55,24.8158274158047)(56,24.8158279491809)(57,24.8158310663297)(58,24.8158343417551)(59,24.9586765164155)(60,24.9586768786324)(61,24.9586770600025)(62,24.9586772425827)(63,24.9586777152887)(64,24.9586781313555)(65,24.9586784441152)(66,24.9586786656845)(67,24.958678953578)(68,24.9586796017419)(69,24.9586797449548)(70,24.958679975847)(71,24.9586801954979)(72,24.9586802848075)(73,24.9586803681567)(74,24.9586805534648)(75,24.9586804768559)(76,24.9586808536046)(77,24.9586810960294)(78,24.9586811148237)(79,24.9586811513208)(80,24.9586812030564)(81,24.958681339382)(82,24.9586810741974)(83,24.9586800245283)(84,24.9586789981636)(85,24.958678232292)(86,24.9586778752884)(87,24.5502385481745)(88,24.5502379881036)(89,24.5502376450463)(90,24.5502373667887)(91,24.5502372688992)(92,24.5502373044198)(93,24.5502372314561)(94,24.5502371408749)(95,24.5502370899152)(96,24.5502370476682)(97,24.5502369957174)(98,24.550237046825)(99,24.5502371022414) 
};
\addlegendentry{$\sigma^{0} = 2.5$}

\definecolor{mygreen}{rgb}{0,0.5.0}
\addplot [
color=mygreen,
mark size=4pt,
only marks,
mark=triangle*,
mark options={solid, fill=mygreen}
]
coordinates{
 (1,26.0668524377089)(2,26.0668524730033)(3,26.0668525081439)(4,26.0668522420029)(5,26.0668524600947)(6,26.0668517962131)(7,26.0668515924353)(8,26.0668512256456)(9,26.0668512842947)(10,26.0668495356863)(11,26.0668489858963)(12,26.0668475277501)(13,26.0668469097513)(14,26.0668451836171)(15,24.3602990935794)(16,24.2180794704684)(17,24.2180657297105)(18,24.2180627542626)(19,24.2180598037546)(20,24.2180578682182)(21,24.2180577781055)(22,24.2180577172146)(23,24.2180573543566)(24,24.2180573050686)(25,24.2180575112814)(26,24.218057997962)(27,24.2180583579628)(28,24.2180593192963)(29,25.1995616128084)(30,25.1995622101949)(31,25.199563658316)(32,25.1995644569417)(33,25.1995740966428)(34,25.1995740714612)(35,25.1995740847423)(36,25.1995733101838)(37,25.19957048471)(38,25.1995679941865)(39,25.1995633851822)(40,25.1995621037761)(41,25.1995606672365)(42,24.3764684473404)(43,24.3764646687601)(44,24.3764630243043)(45,24.3764614837192)(46,24.1914833282238)(47,24.1914817621152)(48,24.191481308264)(49,24.1914809171932)(50,24.1914805670393)(51,24.191479449955)(52,23.3896951280966)(53,23.3896924639953)(54,23.3896923119661)(55,23.3896929372014)(56,23.3896944968132)(57,24.0611160156399)(58,24.0611259297333)(59,25.3578896156744)(60,25.3578900209142)(61,25.3578900149983)(62,25.3578900685506)(63,25.3578910303751)(64,25.3578919103732)(65,25.3578925399044)(66,25.3578929255225)(67,25.3578936293474)(68,25.3578978410362)(69,25.3578981302787)(70,25.3578988914646)(71,25.3578997068341)(72,25.3578999788209)(73,25.3579002986946)(74,25.3579013084398)(75,25.3579010465671)(76,25.3579202464286)(77,25.3579231723978)(78,25.3579235319823)(79,25.3579240830554)(80,25.3579248566017)(81,25.3579270405893)(82,25.3579266252465)(83,25.3579233691572)(84,25.357920854578)(85,25.3579194327854)(86,25.3579190513234)(87,23.852068802152)(88,23.8520679764084)(89,23.8520675582552)(90,23.8520672308287)(91,23.8520674057587)(92,23.8520682754786)(93,23.8520684323494)(94,23.8520684523824)(95,23.8520685786903)(96,23.8520686903288)(97,23.8520687036216)(98,23.8520691252131)(99,23.8520694894524) 
};
\addlegendentry{$\sigma^{0} = 4$}
\end{axis}
\end{tikzpicture}